\documentclass{article} 
\usepackage{latexsym,amsmath,amscd,amssymb,graphics}
\usepackage{amsthm}
\pdfoutput=1 

\usepackage[backend=bibtex,sorting=nyt]{biblatex} 
\DeclareNameAlias{author}{last-first}

\usepackage{enumerate}
\usepackage[toc,page]{appendix}
\usepackage[a4paper,margin=.75in]{geometry}
\usepackage{array}
\usepackage[section]{placeins}
\usepackage{afterpage}
\usepackage{verbatim}

\usepackage{graphicx}
\usepackage{subfig}
\usepackage{framed}
\usepackage[colorlinks]{hyperref}
\usepackage{url}
\usepackage{mathrsfs}
\usepackage{epstopdf}
\usepackage{stmaryrd}

\usepackage[all]{xy}
\usepackage{indentfirst}
\usepackage[mathscr]{euscript}
\usepackage{algorithm}
\usepackage{algpseudocode}
\usepackage{eqparbox}
\usepackage{authblk}

\newdimen{\algindent}
\algnewcommand\LeftComment[2]{%
	\hspace{#1\algindent}$\triangleright$ \eqparbox{COMMENT}{#2} \hfill %
}

\makeatletter
\newcommand*{\skipnumber}[2][0]{%
	{\renewcommand*{\alglinenumber}[1]{}\\ \hspace{#1\algindent}$\triangleright$ \eqparbox{COMMENT}{#2} \hfill}%
	\addtocounter{ALG@line}{-1}}
\makeatother

\algnewcommand{\IIf}[1]{\State\algorithmicif\ #1\ \algorithmicthen}
\algnewcommand{\EndIIf}{\unskip\ \algorithmicend\ \algorithmicif}

\makeatletter
\renewcommand*{\@fnsymbol}[1]{\ensuremath{\ifcase#1\or \dagger\or *\or \ddagger\or
		\mathsection\or \mathparagraph\or \|\or **\or \dagger\dagger
		\or \ddagger\ddagger \else\@ctrerr\fi}}
\makeatother

\makeatletter

\@addtoreset{figure}{section}
\def\thefigure{\thesection.\@arabic\c@figure}
\def\fps@figure{h, t}
\@addtoreset{table}{bsection}
\def\thetable{\thesection.\@arabic\c@table}
\def\fps@table{h, t}
\@addtoreset{equation}{section}

\makeatother


\makeatletter
\AtBeginDocument{%
	\expandafter\renewcommand\expandafter\subsection\expandafter{%
		\expandafter\@fb@secFB\subsection
	}%
}
\makeatother

\newtheorem{theorem}{Theorem}

\newtheorem*{claim*}{Claim}

\newtheorem*{corollary*}{Corollary}

\newtheorem{lemma}[theorem]{Lemma}

\def\bea{\begin{eqnarray}}
\def\eea{\end{eqnarray}}
\def\ba{\begin{array}}
\def\ea{\end{array}}

\let\<\langle
\let \>\rangle

\newcommand{\rem}[1]{}

\newcommand{\bGamma}{\boldsymbol{\Gamma}}

\newcommand{\Continue}{\textbf{continue}}

\newcommand{\argmin}[1]{\underset{#1}{\operatorname{arg}\,\operatorname{min}}\;}

\usepackage[normalem]{ulem}

\graphicspath{ {./} {./figures/}  }

\title{Box Suite Recommendation}
\author{Stuart Rogers\thanks{Email address: \texttt{srogers@umn.edu} or \texttt{stumarcus@gmail.com}}}
\affil{Institute for Mathematics and its Applications, College of Science and Engineering, University of Minnesota, 207 Church ST SE, 306 Lind Hall, Minneapolis, MN 55455, USA}
\date{\today}

\providecommand{\keywords}[1]{\textbf{\textit{Keywords:}} #1}
\begin{document}

\maketitle

\abstract{\noindent 
An algorithm for recommending a suite of boxes for shipping a retailer's online customer orders is presented. }  

\section{Introduction \& Literature Review} \label{sec_intro}
By selecting a cost-optimal suite of boxes for shipping its online customer orders, an online retailer such as Amazon, Walmart, or Target can save a significant amount of money through reduced shipping and material (i.e., corrugate, dunnage, and tape) costs \cite{stroehmer2012repac,stevens2017,wilson2019,mcginty2019}. As a consequence of minimizing cost, the cost-optimal suite also reduces the box outer volume and quantity and weight of material shipped, thereby lowering the online retailer's environmental carbon footprint \cite{stroehmer2012repac,stevens2017,wilson2019,mcginty2019,lu2020user}. An algorithm is presented that recommends such a cost-optimal suite to an online retailer.

The algorithm presented here solves a $p$-median problem \cite{daskin2013networkChap6} for a cost matrix obtained by fitting a statistically significant subset of past historical customer shipments into a large set of candidate boxes. The method to solve the fitting problem presented here is based on mixed integer linear programming (MILP). Reference \cite{dowsland2007simulated} formulates a $p$-median problem to solve the box suite recommendation problem when each box is filled with as many copies of a single product as possible (however a particular box may be used for several different products); \cite{dowsland2007simulated} also assumes that each product is packed into the box in identical layers with the height dimension placed vertically, which makes the fitting problem trivial. Reference \cite{brinker2016optimization} also formulates a $p$-median problem to solve the box suite recommendation problem; however, \cite{brinker2016optimization} uses a heuristic approach to solve the fitting problem. Reference \cite{brinker2016optimization} solves the $p$-median problem via MILP for small problems and via Lagrangian relaxation for large problems.  Unlike previous works, the algorithm presented here modifies the $p$-median problem to lock certain candidate boxes in the suite and handles special fitting constraints, including requirements that certain items be height-oriented and/or bottom-resting when packed. Also, unlike previous works, the algorithm presented here suggests two freely available, robust solvers, POPSTAR \cite{POPSTAR_online} and dc2 \cite{dc2_online}, which are capable of efficiently solving the large $p$-median problems formulated by the algorithm. 

Reference \cite{stroehmer2012repac} solves the box suite recommendation problem by applying an evolutionary algorithm to a statistically significant subset of historical orders, but \cite{stroehmer2012repac} does not formulate the fitting problem and does not formulate a $p$-median problem. Instead of selecting from a large set of candidate boxes, the iterative evolutionary algorithm varies the dimensions of the boxes in the suite in each iteration. Of note, reference \cite{stroehmer2012repac} uses stratified random sampling \cite{levy2013sampling}, instead of simple random sampling, to obtain a very small sample size, just 2,700 sample orders selected from 8 million historical orders; this small sample size enables enormous reductions in computation time compared to the sample size obtained from simple random sampling.

Reference \cite{alonso2016determining} solves the box suite recommendation problem using distributions rather than real order data. Reference \cite{alonso2016determining} assumes that products must be height-oriented and bottom-resting when packed into a box so that the products only occupy a single layer and only a 2D fitting problem needs to be solved. Since products are packed in a single layer, the height of each box in the set of candidate boxes is fixed, so that only the candidate box lengths and widths vary. Reference \cite{alonso2016determining} also assumes that an order must often be split among several boxes during the packing process because the order cannot fit entirely in a single box; in this work, order splits are ignored because the retailer funding this research rarely has to split orders. Reference \cite{alonso2016determining} uses heuristics for the container loading problem and
the bin packing problem to solve the 2D fitting problem.

\section{Notation} \label{sec_notation}
$\mathbb{Z}$ denotes the set of integers, $\mathbb{Z}_2 \equiv \left\{0,1\right\}$ denotes the set of binary numbers, $\mathbb{N}$ denotes the set of natural numbers, which is the same as the set of positive integers, and $\mathbb{N}_0 \equiv \mathbb{N} \cup \left\{0\right\} $ denotes the set of nonnegative integers.  If $m,n \in \mathbb{N}_0$, $\left\{m:n \right\} \equiv \left\{i \in \mathbb{N}_0 \colon m \le i \le n \right\}$ denotes the set of nonnegative integers greater than or equal to $m$ and less than or equal to $n$.  If $N \in \mathbb{N}$, $\llbracket N \rrbracket \equiv \{1,2,\dots,N \} = \left\{ 1:N \right\}$ denotes the set of natural numbers from $1$ to $N$. $\mathbb{R}$ denotes the set of real numbers, $\mathbb{R}_{>0}$ denotes the set of positive real numbers, $\mathbb{R}_{\ge 0} \equiv \mathbb{R}_{>0} \cup \left\{0\right\} $ denotes the set of nonnegative real numbers, and $\mathbb{R}_{\left[0,1\right]} \equiv \left\{w \in \mathbb{R} \colon 0 \le w \le 1\right\} = \mathbb{R} \cap \left[0,1\right]$ denotes the set of real numbers in the closed interval $\left[0,1\right]$. If $N \in \mathbb{N}$, $\left\{ a_n \right\}_{n=1}^N \subset \mathbb{R}$, and $w \in \mathbb{R}$, $\textproc{SearchSortedFirst}\left( \left\{ a_n \right\}_{n=1}^N , w \right)$ finds the index of the first value in $\left\{ a_n \right\}_{n=1}^N$ that is greater than or equal to $w$ and assumes that the values in $\left\{ a_n \right\}_{n=1}^N$ are sorted in nondecreasing order. If $w$ is greater than all values in $\left\{ a_n \right\}_{n=1}^N$, then $N+1$ is returned. The complexity of \textproc{SearchSortedFirst} is $O\left(\log N \right)$ if binary search is used and is $O(1)$ if the algorithm in \cite{cannizzo2018fast} is used. If $a_1,a_2, a_3 \in \mathbb{R}$,
$\textproc{sort} \left( a_1, a_2, a_3 \right)$ returns a permutation $\left(b_1,b_2,b_3 \right)$ of $\left( a_1, a_2, a_3 \right)$ such that $b_1 \ge b_2 \ge b_3$. Having only 3 inputs, the complexity of \textproc{sort} is $O(1)$. Given an array of integers $W \subset \mathbb{N}$ and an integer $i \in \mathbb{N}$, $\textproc{push}\left(W,i\right)$ appends $i$ to the end of the array $W$. The complexity of \textproc{push} is $O(1)$ if an appropriately implemented data structure is utilized to store the array of integers.

\section{Algorithm} \label{sec_algorithm}

\begin{algorithm}
	\caption{Box Suite Recommendation Part I}  \label{alg_bsr}
	\textbf{Input:} Box suite size $p$. $I$ shipments. Each shipment $i \in \llbracket I \rrbracket$ consists of $N_i$ 3D rectangular cartons with outer lengths, widths, and heights $\left\{ \left(p_{in},q_{in},r_{in} \right) \right\}_{n=1}^{N_i}$ and $M_i$ foldable items with outer lengths, widths, and heights $\left\{ \left(s_{im},t_{im},u_{im} \right) \right\}_{m=1}^{M_i}$. $J$ candidate boxes, where $J > p$, sorted by nondecreasing inner volume with inner lengths, widths, and heights $\left\{ \left(x_j,y_j,z_j \right) \right\}_{j=1}^J$ and inner volumes $\left\{ V_j = x_j y_j z_j \right\}_{j=1}^J$. For $j \in \llbracket J-1 \rrbracket$, $ V_j \le V_{j+1}$. A subset $T \subset \llbracket J \rrbracket$  of the candidate boxes, where $\left| T \right| = k \in \left\{0,1,2,\ldots,p-1 \right\}$, must be in the box suite. \\
	\textbf{Output:} A subset $S^* \subset \llbracket J \rrbracket$ of the candidate boxes such that $S^*$ ships all the packable shipments with minimum cost, subject to the constraints $\left| S^* \right| = p$ and $T \subset  S^*$. If such a subset does not exist, then $\O$ is returned. 
	\begin{algorithmic}[1]
	\For {$j=1$ to $J$} \Comment{Iterate over candidate boxes.}
		\State $\left(\tilde x_{j},\tilde y_{j},\tilde z_{j} \right) \gets \textproc{Sort} \left(x_{j},y_{j},z_{j} \right) $ \Comment{Sort box inner dimensions in nonincreasing order.}
		\EndFor
		\skipnumber[0]{Determine into which candidate boxes each candidate box nests.}
		\For {$j=1$ to $J$} \Comment{Iterate over candidate boxes.}
		\State $\Theta_j \gets \left\{j \right\}$ \Comment{$\Theta_j$ stores the set of boxes into which box $j$ nests.}
		\For {$k \gets j+1$ to $J$} \Comment{Iterate over equal or larger volume candidate boxes.}
		\If{$\left(\tilde x_j \le \tilde x_k \right) \wedge \left(\tilde y_j \le \tilde y_k \right) \wedge \left(\tilde z_j \le \tilde z_k \right) $} \Comment{If box $j$ nests inside box $k$.}
		\State $\textproc{push}\left(\Theta_j,k\right)$
		\EndIf
		\EndFor
		\EndFor
		\skipnumber[0]{Construct the $I \times J$ fitting matrix $B$ and determine the packable shipments.}
		\State $\hat I \gets 0$ \Comment{Initialize the number of packable shipments to $0$.}
		\State $W \gets \O$ \Comment{$W$ stores the indices of shipments that are packable.}
		\State $B \gets \mathbf{0}_{I \times J}$ \Comment{Initialize each entry of the fitting matrix to zero.}
		\For {$i=1$ to $I$} \Comment{Iterate over shipments.}
		\State $v_i \gets \sum_{n=1}^{N_i} p_{in} q_{in} r_{in}+\sum_{m=1}^{M_i} s_{im} t_{im} u_{im} $ \Comment{Liquid volume of shipment $i$.}
		\State $j_0 \gets \textproc{SearchSortedFirst}\left( \left\{ V_j \right\}_{j=1}^J , v_i  \right)$ \Comment{Find the smallest box whose inner volume $\ge v_i$.}
		\If{$j_0 = J+1$} \Continue \Comment{This shipment does fit into any box, so skip to the next shipment.}
		\EndIf
		\If{$N_i=0$} \Comment{Only foldable items in the shipment.}
		\State $ B_{i \left\{j_0 : J \right\}} \gets \mathbf{1}_{1 \times \left( J - j_0 + 1 \right)} $
		\State \Continue \Comment{Skip to the next shipment.}
		\EndIf
		\For {$n=1$ to $N_i$} \Comment{Iterate over cartons in shipment $i$.}
		\State $\left(\tilde p_{in},\tilde q_{in},\tilde r_{in} \right) \gets \textproc{Sort} \left(p_{in},q_{in},r_{in} \right) $ \Comment{Sort carton outer dimensions in nonincreasing order.}
		\EndFor
		\State $\mathring p_i \gets \sum_{n=1}^{N_i} \tilde p_{in} \quad \quad \mathring q_i \gets \sum_{n=1}^{N_i} \tilde q_{in} \quad \quad \mathring r_i \gets \sum_{n=1}^{N_i} \tilde r_{in} $
		\State $\hat p_i \gets \max_{1 \le n \le N_i} \tilde p_{in} \quad \quad \hat q_i \gets \max_{1 \le n \le N_i} \tilde q_{in} \quad \quad \hat r_i \gets \max_{1 \le n \le N_i} \tilde r_{in} $
		\For {$j=j_0$ to $J$} \Comment{Iterate over candidate boxes whose inner volume $\ge v_i$.}
		\If{$B_{ij} = 1$} \Continue \Comment{Skip to the next candidate box.}
		\ElsIf{$ \left(\hat p_i \le \tilde x_j \right) \wedge \left( \hat q_i \le \tilde y_j \right) \wedge \left( \hat r_i \le \tilde z_j \right) $} \Comment{Each carton in shipment $i$ must fit in box $j$.}
 \If{$N_i=1 $}  
  $ B_{i \Theta_j} \gets \mathbf{1}_{1 \times \left| \Theta_j \right|} $ \Comment{Only 1 carton in the shipment.}
   \ElsIf{$\left(\mathring p_i \le \tilde x_j \right) \vee \left(\mathring q_i \le \tilde y_j \right) \vee \left(\mathring r_i \le \tilde z_j \right) $} $ B_{i \Theta_j} \gets \mathbf{1}_{1 \times \left| \Theta_j \right|} $ \Comment{Try stacking.}
   \skipnumber[3]{Solve the NP-complete fitting MILP, e.g., with CPLEX or Gurobi.}
  \ElsIf{$\textproc{FittingMILP}\left( \left\{ \left(p_{in},q_{in},r_{in} \right) \right\}_{n=1}^{N_i}  , \left(x_j,y_j,z_j \right) \right)$} $ B_{i \Theta_j} \gets \mathbf{1}_{1 \times \left| \Theta_j \right|} $ 
  \EndIf
  \EndIf
		\EndFor
		\If{$\bigvee_{j=1}^J B_{ij} = 1$} \Comment{Ignore shipments that cannot be packed into any candidate box.}
		\State $\hat I = \hat I+1$ \Comment{$\hat I$ stores the number of packable shipments encountered so far.}
		\State $J_{\hat I} \gets \left\{ j \in \llbracket J \rrbracket \colon B_{ij}=1  \right\} $ \Comment{Find the subset of candidate boxes into which shipment $i$ fits.}
		\State $\textproc{push}\left(W,i\right)$ \Comment{Add shipment $i$ to the set of packable shipments.}
		\EndIf
		\EndFor
	\algstore{myalg}
	\end{algorithmic}
\end{algorithm}

\begin{algorithm}  
\ContinuedFloat
\caption{Box Suite Recommendation Part II}                   
\begin{algorithmic} [1]                   
\algrestore{myalg}
\skipnumber[0]{Construct the $\left(\hat I +k\right) \times J$ cost matrix $C$.}
\State $C \gets \boldsymbol{0}_{\left(\hat I +k\right) \times J}$ \Comment{Preallocate memory for a $\hat I +k$ by $J$ cost matrix.}
		\For {$i =1$ to  $\hat I$} \Comment{Iterate over packable shipments.}
		\For {$j \in J_i$} \Comment{Iterate over the subset of candidate boxes into which packable shipment $i$ fits.}
		\State Compute the cost $C_{ij}$ of shipping packable shipment $i$ (shipment $W_i$) in candidate box $j$. 
		\EndFor
		\EndFor
		\State $\Gamma \gets \left[ \sum_{i=1}^{\hat I} \max_{j \in J_i } C_{ij} \right] +1$ \Comment{Set the penalty cost $\Gamma$ to a sufficiently large positive real number.}
		\For {$i =1$ to  $\hat I$} \Comment{Iterate over packable shipments.}
		\State $C_{i \left[J \setminus J_i \right] } \gets  \bGamma_{1 \times \left( J-\left| J_i \right| \right)} $  \Comment{Packable shipment $i$ ships with penalty cost $\Gamma$ in boxes into which it does not fit.}
		\EndFor
		\For {$i=\hat I +1$ to $\hat I+k$} \Comment{Iterate over fake shipments.}
		\State $C_{i T_{i-\hat I}} \gets 0$ \Comment{Fake shipment $i$ ships for free in box $T_{i-\hat I}$.}
		\State $C_{i \left[ J \setminus \left\{T_{i-\hat I} \right\} \right]} \gets \bGamma_{1 \times \left( J-1 \right)}$ \Comment{Fake shipment $i$ ships with penalty cost $\Gamma$ in all other boxes.}
		\EndFor
		\State $S^* \gets \argmin{\substack{S \subset \llbracket J \rrbracket,\, \left| S \right| = p } } \sum_{i=1}^{\hat I+k} \min_{j \in S } C_{ij}$ \Comment{Solve the NP-hard $p$-median problem, e.g., with POPSTAR or dc2.}
		\State $\Phi \gets \sum_{i=1}^{\hat I+k} \min_{j \in S^* } C_{ij}$ \Comment{Compute the cost of using $S^*$ to ship the packable shipments.}
		\If{$\Phi \ge \Gamma$}
		\State \Return $\O$ \Comment{There is no feasible solution, so return the empty set.}
		\Else
		\State \Return $S^*$ \Comment{Return a cost-optimal suite.}
		\EndIf
\end{algorithmic}
\end{algorithm}

\paragraph{Problem Description and Inputs} An online retailer needs a suite of $p \in \mathbb{N}$ boxes to ship its online customer orders. Moreover, $k \in \{0\} \cup \llbracket p-1 \rrbracket = \left\{0:p-1 \right\}$ boxes in the suite may be prescribed (or locked). For example, the online retailer may need some special boxes for shipping liquid-containing items such as liquid detergent. To save money, the online retailer wants such a suite that minimizes total shipping (shipping plus material) cost. One approach to designing this suite is to select $ I \in \mathbb{N}$ historical customer shipments and $J \in \mathbb{N}$ candidate boxes, where $J>p$. Each historical customer shipment is assigned a unique index $i \in \llbracket I \rrbracket$, and each candidate box is assigned a unique index $j \in \llbracket J \rrbracket$. The set of historical customer shipments should be a small, but statistically significant, randomly sampled subset of the online retailer's past (e.g., within the previous year) online customer shipments. Each historical customer shipment $i \in \llbracket I \rrbracket$ consists of $N_i \in \mathbb{N}_0$ 3D rectangular cartons with positive outer lengths, widths, and heights $\left\{ \left(p_{in},q_{in},r_{in} \right) \right\}_{n=1}^{N_i}$, where $\left(p_{in},q_{in},r_{in} \right) \in \mathbb{R}_{> 0}^3$ for $i \in \llbracket I \rrbracket$ and $n \in \llbracket N_i \rrbracket$, and $M_i \in \mathbb{N}_0$ foldable items with positive outer lengths, widths, and heights $\left\{ \left(s_{im},t_{im},u_{im} \right) \right\}_{m=1}^{M_i}$, where $\left(s_{im},t_{im},u_{im} \right) \in \mathbb{R}_{> 0}^3$ for $i \in \llbracket I \rrbracket$ and $m \in \llbracket M_i \rrbracket$. For the purposes of packing, foldable items are assumed to be liquid so that they may be deformed to fit into arbitrarily-shaped empty spaces inside a box. The set of $J$ candidate boxes should finely discretize the space of all possible boxes and must include the $k$ locked boxes that must be in the suite. The indices of those $k$ locked boxes are prescribed in the subset $T \subset \llbracket J \rrbracket$, where $\left | T \right | = k$. Each candidate box $j \in \llbracket J \rrbracket$ is characterized by a positive inner length, width, and height $\left(x_j,y_j,z_j \right) \in \mathbb{R}_{> 0}^3$. Therefore, the inner volume of candidate box $j \in \llbracket J \rrbracket$ is $V_j = x_j y_j z_j \in \mathbb{R}_{> 0}$. It is assumed that the candidate boxes are sorted by nondecreasing inner volume, which may be realized in $O\left(J \log J \right)$, so that $V_j \le V_{j+1}$ for $j \in \llbracket J-1 \rrbracket$. The optimal suite is obtained by selecting a subset $S \subset \llbracket J \rrbracket$ of the $J$ boxes, such that $\left | S \right | = p$ and $T \subset S$, that ships the $\hat I$ packable shipments, where $\hat I \le I$ (since not all of the $I$ shipments necessarily fit in the $J$ candidate boxes), with minimum cost. Algorithm~\eqref{alg_bsr} gives a method for solving this optimization problem. The next few paragraphs describe the key parts of Algorithm~\eqref{alg_bsr}.

\paragraph{Fitting Matrix} The algorithm begins by sorting each box's dimensions in nonincreasing order. Then, the algorithm determines whether each shipment fits into each candidate box, recording the result in the binary fitting matrix $B \in \mathbb{Z}_2^{I \times J}$. It is simple to solve the fitting problem when there are 0 or 1 cartons in the shipment. There are brute force algorithms for solving the fitting problem for 2 and 3 cartons in the shipment, but these are omitted from the algorithm for conciseness. For 2 or more cartons in the shipment, the algorithm first attempts to stack the cartons along each of the box's three orthogonal axes, to see if they fit. If simple stacking does not work, then the algorithm uses the fitting MILP, a feasibility MILP described in Section~\ref{sec_fitting_MILP}, to solve the fitting problem. Since the fitting MILP must be solved via a third-party MILP solver, which may require a commercial license and which is computationally expensive, substantial run time improvements can be realized by utilizing the brute force fitting algorithms for 2 and 3 cartons. The source code accompanying \cite{martello2007algorithm} provides implementations of the brute force algorithms for 2 and 3 cartons; however, the algorithms presented in that code must be modified to permit rotations. 

Also note that in order for a shipment to fit into a candidate box, the box's inner volume must be greater than or equal to the shipment's liquid volume and each individual carton in the shipment must fit inside the box. For efficiency, the algorithm checks that these necessary conditions are satisfied first before attempting to solve the stacking problem or fitting MILP when there are 2 or more cartons in the shipment. Moreover, if brute force fitting algorithms for 2 and 3 cartons are available, these may be used to check that all pairs and all triples of cartons fit inside the box before attempting to solve the stacking problem or fitting MILP.

\paragraph{Cost Matrix} Next, the algorithm removes shipments that could not be packed into any candidate box, leaving $\hat I \le I$ packable shipments. For each packable shipment $i \in \llbracket \hat I \rrbracket$, $J_i \subset J$ denotes the set of candidate boxes into which packable shipment $i$ fits.  Then, the algorithm constructs the nonnegative cost matrix $C \in \mathbb{R}_{\ge 0}^{ \left(\hat I +k\right) \times J}$ which records the cost of shipping each packable shipment into each candidate box. If packable shipment $i \in \llbracket \hat I \rrbracket$ fits in candidate box $j \in \llbracket J \rrbracket$ (i.e., if $j \in J_i$), the cost $C_{ij} \in \mathbb{R}_{\ge 0}$ to ship packable shipment $i \in \llbracket \hat I \rrbracket$ in candidate box $j \in \llbracket J \rrbracket$ is computed. Note that in order to compute the cost for packable shipment $i \in \llbracket \hat I \rrbracket$, the data for shipment $W_i \in \llbracket I \rrbracket$ is needed; that is, $W_i$ is the shipment index of packable shipment $i$. The data for shipment $W_i \in \llbracket I \rrbracket$ may include the outer dimensions and weights of each item, the shipping carrier (e.g., USPS, FedEx, or UPS), the shipping service (e.g., 1 day, 2 day, or 3 day), and the shipping zone, which is determined by the locations of the shipping store or warehouse and the customer. The cost is computed via a detailed formula, which is omitted here, that depends on the shipping cost charged by the carrier and service combination, the cost of the corrugate used to construct the box, the cost to transport the box blank from the box manufacturer to the ratailer's stores or warehouses, the cost of the dunnage used to fill the empty space between the packed shipment and the box's interior, and the cost of the tape used to seal the top and bottom flaps of the box shut. More simply, if instead of minimizing cost, the online retailer wishes to minimize box outer (or inner) volume shipped or material weight shipped, then $C_{ij}$ is the outer (or inner) volume of box $j$ or the weight of the material (corrugate, dunnage, and tape) used to ship packable shipment $i$ in box $j$, respectively. 

Let $\Gamma \in \mathbb{R}_{> 0}$ be a sufficiently large positive real number such as $\left[ \sum_{i=1}^{\hat I} \max_{j \in J_i } C_{ij} \right] +1 \in \mathbb{R}_{> 0}$. The constant $\Gamma$ serves as a penalty cost to impose constraints on the solution suite $S$. In order to ensure that the solution suite $S$ ships all the packable shipments, $C_{ij}$ is set to $\Gamma$ if  packable shipment $i$ does not fit in candidate box $j$, i.e., if $j \in \llbracket J \rrbracket  \setminus J_i$. In order to force the boxes in $T$ into the solution suite $S$, $k$ fake shipments are appended to the set of packable shipments and $k$ rows are added to the bottom of $C$, where each row, indexed by $i \in \left\{\hat I+1,\hat I+2,\ldots,\hat I+k \right\}$, stores the shipping costs for a fake shipment representing box $T_{i - \hat I}$. The cost of shipping the fake shipment indexed by $i \in \left\{\hat I+1,\hat I+2,\ldots,\hat I+k \right\}$ in box $T_{i - \hat I}$ is 0 and in all other boxes is $\Gamma$. That is, $C_{i T_{i - \hat I}} = 0$ and $C_{i j} = \Gamma$ for $j \in \llbracket J \rrbracket \setminus \left\{ T_{i - \hat I} \right\}$. 

\paragraph{Formulate the $p$-Median Problem}
The optimization problem that must be solved is
\begin{equation} \label{opt_orig}
\argmin{ \substack{T \subset  S \subset \llbracket J \rrbracket,\, \left| S \right| = p, \\ J_i \cap S \ne \O \, \forall i \in \llbracket \hat I \rrbracket } } \sum_{i=1}^{\hat I} \min_{j \in J_i \cap S } C_{ij}.
\end{equation}
Instead, the following optimization problem is solved:
\begin{equation} \label{opt_p_median}
\argmin{\substack{S \subset \llbracket J \rrbracket,\, \left| S \right| = p } } \sum_{i=1}^{\hat I+k} \min_{j \in S  } C_{ij}. 
\end{equation}
The optimization problem \eqref{opt_p_median} is an instance of the $p$-median problem \cite{hakimi1964optimum,hakimi1965optimum,daskin2013networkChap6}, or more precisely the $p$-facility location problem \cite{cornuejols1983uncapicitated}. In the literature, the $p$-median problem is sometimes also called the $k$-median problem \cite{jain1999primal,jain2001approximation,arya2004local}. Given $n \in \mathbb{N}$ customers, $m \in \mathbb{N}$ candidate facilities, $p \in \llbracket m \rrbracket$ candidate facilities to open, and nonnegative costs $d  \in \mathbb{R}_{\ge 0}^{n \times m}$, where $d_{ij} $ is the cost of serving customer $i \in \llbracket n \rrbracket$ with candidate facility $j \in \llbracket m \rrbracket$, the $p$-median problem is to find (or open) a subset $S \subset \llbracket m \rrbracket$ comprised of $p$ candidate facilities that serves all the customers with minimum cost, that is, that minimizes the sum of the costs of serving each customer with its minimum cost open facility. Mathematically, the $p$-median problem is
\begin{equation} \label{opt_p_median_gen}
	\argmin{\substack{S \subset \llbracket m \rrbracket,\, \left| S \right| = p } } \sum_{i=1}^{n} \min_{j \in S  } d_{ij}. 
\end{equation}
To equate \eqref{opt_p_median} to \eqref{opt_p_median_gen}, the shipments are the customers, the boxes are the facilities, $\hat I+k = n$, $J =m$, and $C = d$. 

The equivalence of \eqref{opt_orig} and \eqref{opt_p_median} is established as follows.
\begin{lemma} \label{le1} \eqref{opt_orig} does not have a solution if and only if
\begin{equation}
\min_{\substack{S \subset \llbracket J \rrbracket,\, \left| S \right| = p } } \sum_{i=1}^{\hat I+k} \min_{j \in S  } C_{ij} \ge \Gamma.
\end{equation}	
\end{lemma}
\begin{proof}
 If 
\begin{equation} \label{eq_ineq}
\min_{\substack{S \subset \llbracket J \rrbracket,\, \left| S \right| = p } } \sum_{i=1}^{\hat I+k} \min_{j \in S  } C_{ij} \ge \Gamma,
\end{equation}
then, by construction of $\Gamma$ and $C$,  $\forall S \subset \llbracket J \rrbracket$ such that $ \left| S \right| = p$, $T \not\subset S$ or $\exists i \in \llbracket \hat I \rrbracket \ni J_i \cap S = \O$. To see this explicity, suppose that $\exists S^* \subset \llbracket J \rrbracket$ such that  $ \left| S^* \right| = p$, $T \subset S^*$, and $\forall i \in \llbracket \hat I \rrbracket \; J_i \cap S^* \ne \O$. Then
\begin{equation} \label{eq_just}
\sum_{i=1}^{\hat I+k} \min_{j \in S^*  } C_{ij} = \sum_{i=1}^{\hat I} \min_{j \in S^*  } C_{ij} +\sum_{i=\hat I+1}^{\hat I+k} \min_{j \in S^*  } C_{ij} = \sum_{i=1}^{\hat I} \min_{j \in S^*  } C_{ij} = \sum_{i=1}^{\hat I} \min_{j \in J_i \cap S^*  } C_{ij} \le \sum_{i=1}^{\hat I} \max_{j \in J_i } C_{ij} < \Gamma,
\end{equation}
which contradicts \eqref{eq_ineq}. The second equality in \eqref{eq_just} holds because $T \subset S^*$ implies that $\min_{j \in S^*  } C_{ij} = C_{i T_{i - \hat I}} = 0 \quad \forall i \in \left\{\hat I+1,\hat I+2,\ldots,\hat I+k \right\}$. The third equality in \eqref{eq_just} holds because $S^* = \left( J_i \cap S^* \right) \cup \left( \left(\llbracket J \rrbracket \setminus J_i \right) \cap S^* \right)$, $C_{ij} < \Gamma \; \forall j \in J_i $, $C_{ij} = \Gamma \; \forall j \in \llbracket J \rrbracket \setminus J_i $,  and $J_i \cap S^* \ne \O \; \forall i \in \llbracket \hat I \rrbracket$. In this case, there does not exist a solution to \eqref{opt_orig}. 

Conversely, if there does not exist a solution to \eqref{opt_orig}, then $\forall S \subset \llbracket J \rrbracket$ such that $ \left| S \right| = p$, $T \not\subset S$ or $\exists i \in \llbracket \hat I \rrbracket \ni J_i \cap S = \O$. Therefore, by construction of $C$, $\forall S \subset \llbracket J \rrbracket$ such that $ \left| S \right| = p$, $\exists i \in \llbracket \hat I+k \rrbracket \ni \min_{j \in S  } C_{ij} = \Gamma$. Hence,  $\forall S \subset \llbracket J \rrbracket$ such that $ \left| S \right| = p$,
\begin{equation}
\sum_{i=1}^{\hat I+k} \min_{j \in S  } C_{ij} \ge \Gamma,
\end{equation}
so that
\begin{equation}
\min_{\substack{S \subset \llbracket J \rrbracket,\, \left| S \right| = p } } \sum_{i=1}^{\hat I+k} \min_{j \in S  } C_{ij} \ge \Gamma.
\end{equation}
\end{proof}
\noindent The contrapositive of Lemma~\eqref{le1} gives the following corollary.
\begin{corollary*} \eqref{opt_orig} has a solution if and only if
	\begin{equation}
	\min_{\substack{S \subset \llbracket J \rrbracket,\, \left| S \right| = p } } \sum_{i=1}^{\hat I+k} \min_{j \in S  } C_{ij} < \Gamma.
	\end{equation}
\end{corollary*}
\begin{lemma} If 
\begin{equation} \label{eq_claim_condition}
\min_{\substack{S \subset \llbracket J \rrbracket,\, \left| S \right| = p } } \sum_{i=1}^{\hat I+k} \min_{j \in S  } C_{ij} < \Gamma,
\end{equation}
that is, if \eqref{opt_orig} has a solution, then  
\begin{equation} \label{eq_cost_simp}
\argmin{ \substack{T \subset  S \subset \llbracket J \rrbracket,\, \left| S \right| = p, \\ J_i \cap S \ne \O \, \forall i \in \llbracket \hat I \rrbracket } } \sum_{i=1}^{\hat I} \min_{j \in J_i \cap S } C_{ij} =
\argmin{\substack{S \subset \llbracket J \rrbracket,\, \left| S \right| = p } } \sum_{i=1}^{\hat I+k} \min_{j \in S  } C_{ij}. 
\end{equation} 
\end{lemma}
\begin{proof}
If the inequality \eqref{eq_claim_condition} holds, then, by construction of $C$, $\forall S^* \in \argmin{\substack{S \subset \llbracket J \rrbracket,\, \left| S \right| = p } } \sum_{i=1}^{\hat I+k} \min_{j \in S  } C_{ij}$ the following properties hold:
\begin{enumerate}[(a)]
 \item  $T \subset S^*$ 
 \item  $J_i \cap S^* \ne \O \; \forall i \in \llbracket \hat I \rrbracket $.
 \end{enumerate}
 Consequently,
 \begin{equation} \label{eq_der1}
 \begin{split}
 \argmin{\substack{S \subset \llbracket J \rrbracket,\, \left| S \right| = p } } \sum_{i=1}^{\hat I+k} \min_{j \in S  } C_{ij} &= \argmin{ \substack{T \subset  S \subset \llbracket J \rrbracket,\, \left| S \right| = p, \\ J_i \cap S \ne \O \, \forall i \in \llbracket \hat I \rrbracket } } \sum_{i=1}^{\hat I+k} \min_{j \in S  } C_{ij} \\
 &= \argmin{ \substack{T \subset  S \subset \llbracket J \rrbracket,\, \left| S \right| = p, \\ J_i \cap S \ne \O \, \forall i \in \llbracket \hat I \rrbracket } } \left[ \sum_{i=1}^{\hat I} \min_{j \in S  } C_{ij} +\sum_{i=\hat I+1}^{\hat I+k} \min_{j \in S  } C_{ij}\right]  \\
 &= \argmin{ \substack{T \subset  S \subset \llbracket J \rrbracket,\, \left| S \right| = p, \\ J_i \cap S \ne \O \, \forall i \in \llbracket \hat I \rrbracket } } \sum_{i=1}^{\hat I} \min_{j \in S  } C_{ij}   \\
 &= \argmin{ \substack{T \subset  S \subset \llbracket J \rrbracket,\, \left| S \right| = p, \\ J_i \cap S \ne \O \, \forall i \in \llbracket \hat I \rrbracket } } \sum_{i=1}^{\hat I} \min_{j \in J_i \cap S } C_{ij}.
 \end{split}
 \end{equation}
 The first equality follows from properties (a) and (b). The third equality holds because $T \subset S$ implies that $\min_{j \in S  } C_{ij} = C_{i T_{i - \hat I}} = 0 \quad \forall i \in \left\{\hat I+1,\hat I+2,\ldots,\hat I+k \right\}$. The fourth equality holds because $S = \left( J_i \cap S \right) \cup \left( \left(\llbracket J \rrbracket \setminus J_i \right) \cap S \right)$, $C_{ij} < \Gamma \; \forall j \in J_i $, $C_{ij} = \Gamma \; \forall j \in \llbracket J \rrbracket \setminus J_i $,  and $J_i \cap S \ne \O \; \forall i \in \llbracket \hat I \rrbracket$. 
\end{proof}
In summary, if \eqref{opt_p_median} does not have a solution with cost less than $\Gamma$, then \eqref{opt_orig} does not have a solution, and if \eqref{opt_p_median} does have a solution with cost less than $\Gamma$, then \eqref{opt_orig} has a solution and the solution set realized by \eqref{opt_orig} equals that realized by \eqref{opt_p_median}. 

\paragraph{Solving the $p$-Median Problem} The $p$-median problem is NP-hard \cite{kariv1979algorithmic}. There are many methods to solve the $p$-median problem including MILP, Lagrangian relaxation, heuristics (e.g., myopic algorithm, neighborhood search,
exchange algorithm, Lin--Kernighan neighborhood exchange algorithm), metaheuristics (e.g., simulated annealing, genetic
algorithm, tabu search, heuristic concentration, variable neighborhood
search, ant colony optimization, scatter search, and GRASP), and hyper-heuristics \cite{daskin2013networkChap6,daskin2015p,kwon2019juliaChap10,reese2006solution,mladenovic2007p,ren2010ant}. 

The exchange algorithm \cite{teitz1968heuristic} is a simple heuristic which starts with a feasible solution $S$ comprised of $p$ candidate facilities and finds a pair of facilities, $a \in S$ and $b \in S^{\mathsf{c}}$, such that the new feasible solution $\left( S \setminus \left\{a\right\} \right) \cup \left\{b\right\}$ decreases the cost of serving all the customers compared to the original feasible solution $S$. This exchange or swapping procedure is repeated iteratively until no further improvement is realized, depending on the definition of the local neighborhood about a feasible solution. There are several variations of the exchange algorithm which search over different local neighborhoods, including the first improvement \cite{whitaker1983fast}, best improvement \cite{hansen1997variable}, and Lin--Kernighan \cite{kochetov2005large} neighborhoods. Moreover, there are several efficient implementations of the exchange algorithm \cite{whitaker1983fast,densham1992strategies,hansen1997variable,resende2007fast} which provide significant computational savings compared to a na\"{\i}ve implementation. Due to its simplicity, the exchange algorithm is only effective at solving very small instances of the $p$-median problem. However, the exchange algorithm is often used as a local search algorithm by other more complicated solution methods, such as Lagrangian relaxation and metaheuristics.

Given $n \in \mathbb{N}$ customers, $m \in \mathbb{N}$ candidate facilities, $p \in \llbracket m \rrbracket$ candidate facilities to open, and nonnegative costs $d \in \mathbb{R}_{\ge 0}^{n \times m}$, where $d_{ij} $ is the cost of serving customer $i \in \llbracket n \rrbracket$ with candidate facility $j \in \llbracket m \rrbracket$, the MILP formulation of the $p$-median problem \eqref{opt_p_median_gen} is \cite{daskin2015p,kwon2019juliaChap10}
\begin{equation} \label{opt_p_median_milp}
	\begin{split} 
		\min_{x,y} \sum_{i =1}^n &\sum_{j = 1}^m d_{i j}  x_{ij} \quad \ni \quad \textrm{ {\tiny - minimize the cost of serving all the customers subject to the following constraints:}} \\
		& y_j \in \{0,1\} \quad \forall j \in \llbracket m \rrbracket, \quad \textrm{ {\tiny - whether candidate facility $j$ is open}} \\
		& x_{ij} \ge 0 \quad \forall i \in \llbracket n \rrbracket \quad \forall j \in \llbracket m \rrbracket, \quad \textrm{ {\tiny - whether customer $i$ is served by candidate facility $j$}} \\
		& \sum_{j=1}^m y_j = p, \quad \textrm{ {\tiny - open $p$ candidate facilities}} \\
		& \sum_{j =1}^m x_{ij} = 1 \quad \forall i \in \llbracket n \rrbracket, \quad \textrm{ {\tiny - each customer must be served by exactly one candidate facility}} \\
		& x_{ij} \le y_j \quad \forall i \in \llbracket n \rrbracket \quad \forall j \in \llbracket m \rrbracket. \quad \textrm{ {\tiny - if candidate facility $j$ is closed, no customer can be served by it}}
	\end{split}
\end{equation}
In \eqref{opt_p_median_milp}, $y \in \mathbb{Z}_2^m$ determines which candidate facilities are open and $x \in \mathbb{R}_{\left[0,1\right]}^{n \times m}$ determines which open facility is assigned to each customer. \eqref{opt_p_median_milp} formulates $x \in \mathbb{R}_{\left[0,1\right]}^{n \times m}$ instead of $x \in \mathbb{Z}_2^{n \times m}$, because the former imposes far fewer integrality constraints and therefore is much more computationally tractable for a MILP solver. However, in the solution of \eqref{opt_p_median_milp}, $x_{ij}$ may not be 0 or 1 if customer $i \in \llbracket n \rrbracket$ may be served by more than one open facility with the same minimum cost. It is trivial to transform a solution $x \in \mathbb{R}_{\left[0,1\right]}^{n \times m}$ of \eqref{opt_p_median_milp} into an equivalent solution $\tilde x \in \mathbb{Z}_2^{n \times m}$ as follows.  
Let $S \equiv \left\{ j \in \llbracket m \rrbracket \colon y_j = 1 \right\} \subset \llbracket m \rrbracket$ denote the open facilities in the solution of \eqref{opt_p_median_milp}. Then, for each  $i \in \llbracket n \rrbracket$, let 
\begin{equation}
	j_i \equiv \argmin{k \in S} d_{ik},
\end{equation} 
breaking ties arbitrarily if the minimum cost of serving customer $i$ over the open facilities in $S$ is not unique. Then for each  $i \in \llbracket n \rrbracket$ and for each $j \in \llbracket m \rrbracket$, $\tilde x \in \mathbb{Z}_2^{n \times m}$ is constructed via:
\begin{equation} \label{eq_tilde_x}
	\tilde x_{ij} = \begin{cases}   
		1  &\textrm{if} \; j=j_i, \\ 
		0 &\textrm{if} \; j \in \llbracket m \rrbracket \setminus \left\{j_i \right\}.
	\end{cases}
\end{equation} 
A third-party MILP solver, such as Gurobi \cite{Gurobi_online}, CPLEX \cite{CPLEX_online}, or CBC \cite{CBC_online}, must be used to solve the MILP \eqref{opt_p_median_milp}; the reader is referred to ``Solving the Fitting MILP" in Section~\ref{sec_fitting_MILP} for more information about the particular MILP solvers mentioned here.  However, MILP solvers are only able to solve fairly small instances of the $p$-median problem to within a prescribed absolute or relative optimality gap of the globally optimal solution. Because MILP solvers rely on branch-and-bound algorithms, if an MILP solver is able to find a solution (not necessarily the globally optimal solution), it is also able to provide a lower bound of the globally optimal solution via the absolute or relative optimality gap.  

The objective function $\mathcal{J}$ that is minimized in \eqref{opt_p_median_milp} is $\mathcal{J}\left(x\right) \equiv \sum_{i =1}^n \sum_{j = 1}^m d_{i j}  x_{ij}$. Construct the Lagrangian function $\mathcal{L}$ by adjoining the fourth constraint, $\sum_{j =1}^m x_{ij} = 1 \quad \forall i \in \llbracket n \rrbracket$, in \eqref{opt_p_median_milp} to the objective function $\mathcal{J}$ via the vector of Lagrange multipliers $\lambda \in \mathbb{R}^n$: 
\begin{equation} \label{eq_lagrangian}
\begin{split}
	\mathcal{L}\left(x,\lambda \right) \equiv \sum_{i =1}^n \sum_{j = 1}^m d_{i j}  x_{ij}+\sum_{i=1}^n \lambda_i \left( 1- \sum_{j =1}^m x_{ij} \right) &= \sum_{i =1}^n \sum_{j = 1}^m \left( d_{i j} - \lambda_i \right)  x_{ij}+\sum_{i=1}^n \lambda_i \\ &= \sum_{j = 1}^m \sum_{i =1}^n \left( d_{i j} - \lambda_i \right)  x_{ij}+\sum_{i=1}^n \lambda_i.
\end{split} 
\end{equation}
 The primal form of the MILP formulation \eqref{opt_p_median_milp} of the $p$-median problem is
\begin{equation} \label{opt_primal}
	\begin{split} 
		\min_{x,y} \max_{\lambda} \; &\mathcal{L}\left(x,\lambda \right)  \quad \ni \\
		& y_j \in \{0,1\} \quad \forall j \in \llbracket m \rrbracket,  \\
		& x_{ij} \in \{0,1\} \quad \forall i \in \llbracket n \rrbracket \quad \forall j \in \llbracket m \rrbracket,  \\
		& \sum_{j=1}^m y_j = p,  \\
		& x_{ij} \le y_j \quad \forall i \in \llbracket n \rrbracket \quad \forall j \in \llbracket m \rrbracket. 
	\end{split}
\end{equation}
The primal problem \eqref{opt_primal} is equivalent to \eqref{opt_p_median_milp}. Now construct the dual function $\mathcal{D}$:
\begin{equation} \label{eq_dual}
	\begin{split} 
		\mathcal{D}\left(\lambda\right) \equiv \min_{x,y} \; &\mathcal{L}\left(x,\lambda \right)  \quad \ni \\
		& y_j \in \{0,1\} \quad \forall j \in \llbracket m \rrbracket,  \\
		& x_{ij} \in \{0,1\} \quad \forall i \in \llbracket n \rrbracket \quad \forall j \in \llbracket m \rrbracket,  \\
		& \sum_{j=1}^m y_j = p,  \\
		& x_{ij} \le y_j \quad \forall i \in \llbracket n \rrbracket \quad \forall j \in \llbracket m \rrbracket.
	\end{split}
\end{equation}
The following are important properties of the dual function $\mathcal{D}$.
\begin{enumerate}
	\item For fixed $\lambda$, it is trivial to construct $x \in \mathbb{Z}_2^{n \times m}$ and $y \in \mathbb{Z}_2^m$ that minimize the Lagrangian $\mathcal{L}$ subject to satisfying the constraints in \eqref{eq_dual} \cite{daskin2013networkChap6,daskin2015p,kwon2019juliaChap10}. Therefore, it is simple to solve the constrained Lagrangian minimization problem and compute $\mathcal{D}\left(\lambda\right)$.
	\item For fixed $\lambda$, the solution $x \in \mathbb{Z}_2^{n \times m}$ and $y \in \mathbb{Z}_2^m$ of $\mathcal{D}\left(\lambda\right)$ may be easily transformed into a corresponding primal solution, i.e., a feasible solution of the $p$-median problem \cite{daskin2013networkChap6,daskin2015p,kwon2019juliaChap10}.
	\item For fixed $\lambda$, $\mathcal{D}\left(\lambda\right)$ is a lower bound of the solution to the primal problem \eqref{opt_primal} \cite{conforti2014integer}.  
	\item $\mathcal{D}$ is concave as a function of $\lambda$ \cite{conforti2014integer}.
	\item $\mathcal{D}$ is piecewise linear, and is therefore nondifferentiable, as a function of $\lambda$ \cite{conforti2014integer}.
\end{enumerate}
The dual form of the MILP formulation \eqref{opt_p_median_milp} of the $p$-median problem is
\begin{equation} \label{opt_dual}
\begin{split}
 \max_{\lambda} \mathcal{D}\left(\lambda\right) = \max_{\lambda} \min_{x,y}  \; &\mathcal{L}\left(x,\lambda \right) \quad \ni \\
 & y_j \in \{0,1\} \quad \forall j \in \llbracket m \rrbracket,  \\
 & x_{ij} \in \{0,1\} \quad \forall i \in \llbracket n \rrbracket \quad \forall j \in \llbracket m \rrbracket,  \\
 & \sum_{j=1}^m y_j = p, \\
 & x_{ij} \le y_j \quad \forall i \in \llbracket n \rrbracket \quad \forall j \in \llbracket m \rrbracket.
\end{split}
\end{equation}
Since the dual function $\mathcal{D}$ is concave and nondifferentiable, the dual problem \eqref{opt_dual} may be solved by subgradient optimization \cite{daskin2013networkChap6,daskin2015p,kwon2019juliaChap10}, which is a particular method of convex optimization. The solution to the dual problem \eqref{opt_dual} is a lower bound to the solution of the  primal problem \eqref{opt_primal} \cite{conforti2014integer}; in some cases, the solution to the dual problem \eqref{opt_dual} may even coincide with the solution to the primal problem \eqref{opt_primal}. In the Lagrangian relaxation approach to solving the $p$-median problem \eqref{opt_p_median_gen}, the dual problem \eqref{opt_dual} is solved via iterative subgradient optimization, starting from an initial guess $\lambda$ of the Lagrange multipliers. In each iteration of the subgradient optimization, 1) a lower bound is constructed by solving the constrained Lagrangian minimization problem and computing the dual function $\mathcal{D}$ in \eqref{eq_dual} for the current estimate $\lambda$ of the Lagrange multipliers, 2) a corresponding upper bound is constructed from the lower bound, and 3) the current estimate $\lambda$ of the Lagrange multipliers is updated. If in a given iteration, a new smallest upper bound is found, it may be further improved via a heuristic such as the exchange algorithm. As described in \cite{daskin2015p}, if in a given iteration, a new smallest upper bound or a new largest lower bound is found, candidate facilities can be forced into and out of the optimal solution, which can reduce the computational time of Lagrangian relaxation. By keeping track of the smallest upper bound and the largest lower bound found over all iterations, Lagrangian relaxation provides both upper and lower bounds of the globally optimal solution to the $p$-median problem. The iterations continue until some stopping criterion is satisfied, such as executing a maximum number of iterations or realizing a prescribed absolute or relative optimality gap. Instead of adjoining the fourth constraint, $\sum_{j =1}^m x_{ij} = 1 \quad \forall i \in \llbracket n \rrbracket$, in \eqref{opt_p_median_milp} to the objective function $\mathcal{J}$, it is possible to adjoin the fifth constraint, $x_{ij} \le y_j \quad \forall i \in \llbracket n \rrbracket \quad \forall j \in \llbracket m \rrbracket$, in \eqref{opt_p_median_milp} to the objective function $\mathcal{J}$ via a matrix of nonnegative Lagrange multipliers $\lambda \in \mathbb{R}_{\ge 0}^{n \times m}$; this alternative approach is discussed in \cite{daskin2013networkChap6}.

POPSTAR \cite{POPSTAR_online} is a freely available $p$-median problem solver implemented in C++. POPSTAR solves the $p$-median problem via a hybrid metaheuristic that combines GRASP with path-relinking and the genetic algorithm \cite{resende2004hybrid}. Moreover, POPSTAR performs local searches via a fast implementation of the exchange algorithm \cite{resende2007fast}. Reference \cite{mladenovic2007p}, published in 2007, comprehensively surveyed many methods, excluding hyper-heuristics, for solving the $p$-median problem and concluded that the overall algorithm implemented by POPSTAR is the best. dc2 \cite{dc2_online} is a recent, freely available $p$-median problem solver implemented in C that is still under development. dc2 solves the $p$-median problem via several metaheuristics, including GRASP, path-relinking, and a new metaheuristic called disperse construction, and performs local searches via several fast implementations of the exchange algorithm \cite{whitaker1983fast,hansen1997variable,resende2007fast}. Unlike POPSTAR, dc2 is multithreaded and therefore is able to exploit the parallelism offered by multicore CPUs. POPSTAR and dc2 can solve instances of the $p$-median problem with several thousand customers, several thousand candidate facilities, and $p \le 20$ in a few minutes on a modern laptop.  Several metaheuristics, including GRASP and the genetic algorithm, have been implemented on GPGPUs to solve the $p$-median problem \cite{santos2010parallel,albdaiwi2017gpu}.

\paragraph{Validation} The box suite recommendation may be validated by packing the optimization shipment set and a much larger randomly sampled shipment set into the recommended suite $S^*$, where each shipment is assigned to the minimum cost box in the suite into which it fits. Several metrics, such as percentage of shipments packed into each box, percentage of total cost shipped by each box, percentage of all box outer volume shipped by each box, and percentage liquid void space, collected for both packings can be compared. If the metrics are similar, then the cost savings afforded by the recommended suite $S^*$ should be expected to hold on all shipments. An alternative validation method is to pack the optimization shipment set and a much larger randomly sampled shipment set into several suites $Q \cup S^*$, where each suite $S \in Q$ satisfies $T \subset S$, $\left| S \right| = p$, and $J_i \cap S \ne \O \, \forall i \in \llbracket \hat I \rrbracket $, and ensure that the cost reductions (comparing the cost of each suite $S \in Q$ to the cost of $S^*$) predicted by the optimization shipment set agree with those predicted by the large shipment set. If the various metrics or cost reductions are dissimilar, then Algorithm~\eqref{alg_bsr} should be run again using a larger set of randomly sampled historical customer shipments.

\paragraph{Fine-Tuning} The recommended suite $S^*$ can be further refined (or fine-tuned) by running Algorithm~\eqref{alg_bsr} again on a new set of candidate boxes $J'$ obtained by taking small variations above and below the inner lengths, widths, and heights of the unlocked boxes $S^* \setminus T$ in the recommended suite $S^*$. Like the original set of candidate boxes $J$, the new set of candidate boxes $J'$ must also include the $k$ locked boxes.

\paragraph{Modifications to Handle Height-Oriented and Bottom-Resting Cartons} Some shipments may contain cartons that must be packed vertically, so that their height dimensions must be parallel to the box's height dimension. Such cartons will be called height-oriented (HO). This constraint is quite common in online retail, e.g., liquid detergent often must be HO when packed into a shipping box to prevent spillage. When packed into a box, a HO carton may be rotated in only 2 (instead of 6) possible ways. In order to handle this additional constraint, Algorithm~\eqref{alg_bsr} must be modified in the following ways. Lines 1-11 in Algorithm~\eqref{alg_bsr} must be replaced with the pseudocode given in Algorithm~\eqref{alg_bsr_mod1}, and lines 24-26 in Algorithm~\eqref{alg_bsr} must be replaced with the pseudocode given in Algorithm~\eqref{alg_bsr_mod2}. HO cartons may have the additional constraint that they rest at the bottom of the box, to encourage stability and mitigate the possibility of tipping over. To handle this additional constraint, the Boolean condition in line 33 in Algorithm~\eqref{alg_bsr}, which checks to see if the cartons can be stacked along any of the 3 box dimensions, must be replaced with the Boolean condition
\begin{equation} \label{bc_stack_mod}
\left(\mathring p_i \le \tilde x_j \right) \vee \left(\mathring q_i \le \tilde y_j \right) \vee \left(H_i \le 1 \wedge \mathring r_i \le \tilde z_j \right),
\end{equation}
where $H_i$ (see lines 1 and 5 in Algorithm~\eqref{alg_bsr_mod2}) denotes the number of HO cartons in shipment $i$, since stacking the cartons along the box's height dimension is valid only if there are less than 2 HO cartons in the shipment. More generally, only a proper subset of HO cartons may need to rest at the bottom of the box or some  non-HO cartons may need to rest at the bottom of the box. Such cartons will be called bottom-resting (BR). To handle this more general case, let $R_i$ denote the number of BR cartons in shipment $i$. Then the Boolean condition in line 33 in Algorithm~\eqref{alg_bsr} must be replaced with the Boolean condition
\begin{equation} \label{bc_stack_mod2}
\left(\mathring p_i \le \tilde x_j \right) \vee \left(\mathring q_i \le \tilde y_j \right) \vee \left(R_i \le 1 \wedge \mathring r_i \le \tilde z_j \right),
\end{equation}
since stacking the cartons along the box's height dimension is valid only if there are less than 2 BR cartons in the shipment. \eqref{bc_stack_mod} is valid instead of \eqref{bc_stack_mod2} only if all HO cartons are BR and there are no non-HO cartons that are BR, in which case $H_i = R_i$. The paragraph ``Special Packing Constraints" in Section~\ref{sec_fitting_MILP} discusses how to enforce HO and BR packing constraints in the fitting MILP.

\begin{algorithm}
	\caption{Box Suite Recommendation: HO Modification I}  \label{alg_bsr_mod1} 
	\begin{algorithmic}[1]
		\For {$j=1$ to $J$} \Comment{Iterate over candidate boxes.}
		\skipnumber[1]{Sort the length and width box inner dimensions in nonincreasing order for shipments with HO cartons.}
		\State $\left(\bar x_{j},\bar y_{j} \right) \gets \textproc{Sort} \left(x_{j},y_{j} \right) \quad \quad \bar z_{j} \gets  z_j$
		\skipnumber[1]{Sort box inner dimensions in nonincreasing order for shipments with no HO cartons.}
		\State $\left(\ddot x_{j},\ddot y_{j},\ddot z_{j} \right) \gets \textproc{Sort} \left(x_{j},y_{j},z_{j} \right) $ 
		\EndFor
		\skipnumber[0]{Determine into which candidate boxes each candidate box nests.}
		\For {$j=1$ to $J$} \Comment{Iterate over candidate boxes.}
		\skipnumber[1]{$\Psi_j$ stores the set of boxes into which box $j$ nests, only permitting the 2 HO rotations for nesting.}
		\State $\Psi_j \gets \left\{j \right\}$ 
		\skipnumber[1]{$\Sigma_j$ stores the set of boxes into which box $j$ nests, permitting any of the 6 possible rotations for nesting.}
		\State $\Sigma_j \gets \left\{j \right\}$
		\For {$k \gets j+1$ to $J$} \Comment{Iterate over equal or larger volume candidate boxes.}
		\If{$\left(\bar x_j \le \bar x_k \right) \wedge \left(\bar y_j \le \bar y_k \right) \wedge \left(\bar z_j \le \bar z_k \right) $} \Comment{If box $j$ nests inside box $k$ in a HO way.}
		\State $\textproc{push}\left(\Psi_j,k\right)$
		\EndIf
		\If{$\left(\ddot x_j \le \ddot x_k \right) \wedge \left(\ddot y_j \le \ddot y_k \right) \wedge \left(\ddot z_j \le \ddot z_k \right) $} \Comment{If box $j$ nests inside box $k$.}
		\State $\textproc{push}\left(\Sigma_j,k\right)$
		\EndIf
		\EndFor
		\EndFor
\end{algorithmic}
\end{algorithm}

\begin{algorithm}
	\caption{Box Suite Recommendation: HO Modification II}  \label{alg_bsr_mod2} 
	\begin{algorithmic}[1]
\State $H_i \gets 0$ \Comment{$H_i$ records the number of HO cartons in shipment $i$.}
\For {$n=1$ to $N_i$} \Comment{Iterate over cartons in shipment $i$.}
\If{carton $n$ is HO} \Comment{If carton $n$ must be HO when packed into a box.}
\skipnumber[2]{Sort carton length and width outer dimensions in nonincreasing order for HO packing.}
\State $\left(\tilde p_{in},\tilde q_{in} \right) \gets \textproc{Sort} \left(p_{in},q_{in} \right) \quad \quad  \tilde r_{in} \gets r_{in}$ 
\State $H_i \gets H_i+1 $ \Comment{Increment $H_i$.}
\Else \Comment{If carton $n$ need not be HO when packed into a box.}
\skipnumber[2]{Sort carton outer dimensions in nonincreasing order for non-HO packing.}
\State $\left(\tilde p_{in},\tilde q_{in},\tilde r_{in} \right) \gets \textproc{Sort} \left(p_{in},q_{in},r_{in} \right)$
\EndIf
\EndFor
\If{$H_i > 0$} \Comment{If shipment $i$ contains a HO carton.}
\State $\left\{\Theta_j\right\}_{j=j_0}^J \gets \left\{\Psi_j\right\}_{j=j_0}^J \quad \quad \left\{\left(\tilde x_{j},\tilde y_{j},\tilde z_{j} \right)\right\}_{j=j_0}^J \gets \left\{\left(\bar x_{j},\bar y_{j},\bar z_{j} \right)\right\}_{j=j_0}^J $ 
\Else \Comment{If shipment $i$ does not contain a HO carton.}
\State $\left\{\Theta_j\right\}_{j=j_0}^J \gets \left\{\Sigma_j\right\}_{j=j_0}^J \quad \quad \left\{\left(\tilde x_{j},\tilde y_{j},\tilde z_{j} \right)\right\}_{j=j_0}^J \gets \left\{\left(\ddot x_{j},\ddot y_{j},\ddot z_{j} \right)\right\}_{j=j_0}^J $ 
\EndIf
\end{algorithmic}
\end{algorithm}

\paragraph{Numerical Experiments} Three 10 box suites are constructed according to Algorithm~\eqref{alg_bsr}. In this description of the numerical experiments performed to obtain the suites, the inner dimensions of a box and outer dimensions of an item are always assumed to be sorted in nonincreasing order. The suites are selected from a set of 5,284 candidate boxes, which is the set of boxes with integral sorted inner dimensions $\left\{\left(x,y,z \right) \in \mathbb{N}^3 \colon x \ge y \ge z, 40 \ge x \ge 5, 20 \ge y \ge 4, 16 \ge z \ge 1 \right\}$ and where the smallest inner volume box has sorted inner dimensions $\left(5,4,1\right)$ and the largest inner volume box has sorted inner dimensions $\left(40,20,16\right)$. The first suite does not lock any boxes, the second suite locks the box with sorted inner dimensions $\left(12,7,6\right)$, and the third suite locks the two boxes with sorted inner dimensions $\left(12,7,6\right)$ and $\left(16,12,6\right)$. 

The suites are optimized based on 15,000 shipments, where each shipment is comprised of items selected from a set of 2,324 candidate items. Each item is assumed to be a 3D rectangular carton with integral outer dimensions, so that foldable items are excluded. No two items have the same sorted outer dimensions. Moreover, none of the items have special packing constraints, such as having to be height-oriented or bottom-resting. Each shipment is comprised of 1 to 8 unique items, where each unique item has quantity 1 to 8 such that there are no more than 8 total cartons in the shipment. 

The set of candidate boxes, the set of candidate items, and the set of shipments are available on Mendeley Data in three separate \texttt{CSV} files, \texttt{boxes.csv}, \texttt{items.csv}, and \texttt{shipments.csv}, respectively \cite{https://doi.org/10.17632/f2bnnnm5zc.3}. Each candidate box occupies a row in \texttt{boxes.csv} comprised of a unique integral box ID and three integral inner dimensions sorted in nonincreasing order. Each candidate item occupies a row in \texttt{items.csv} comprised of a unique integral item ID and three integral outer dimensions sorted in nonincreasing order. Each row in \texttt{shipments.csv} represents all or part of a shipment and is comprised of a unique integral shipment ID, an item's integral ID, the item's quantity, and the item's outer dimensions sorted in nonincreasing order. All the items in a shipment are assigned the same unique integral shipment ID and therefore can be grouped together based on it.

The 15,000 shipments, consisting of 1 to 8 cartons (i.e., no foldable items), were fit into the 5,284 candidate boxes. In all, it took 1.4 hours to solve all the fitting problems. Of the 15,000 shipments, 20.88\% (3,132 shipments) consist of 4 or more cartons, for which the fitting MILP \eqref{eq_orient}-\eqref{eq_lbb}  was solved by CPLEX v12.10.0 since brute force fitting algorithms were used for shipments consisting of 1, 2, or 3 cartons. 173,066 fitting MILPs had to be solved. MILPs taking more than 5 seconds (s) to solve were terminated early, in which case the underlying shipment was declared to not fit into the underlying candidate box. Of the 173,066 fitting MILPs, 29,618 were feasible (a fit was found), 143,393 were infeasible (no fit possible), and 55 were terminated early due to the 5[s] time limit. Of the 15,000 shipments, 112 shipments did not fit into any candidate box, so that only 14,888 shipments are packable. The reader is referred to ``Benchmarking I" in Section~\ref{sec_fitting_MILP} for a performance comparison of different MILP solvers on various formulations of the fitting MILP.  

The cost of shipping a shipment in a particular candidate box is the box's inner volume, so that each optimal suite is a set of 10 boxes that ships all the packable shipments with minimum total inner volume shipped, subject to locking no, one, or two boxes.  Given the cost matrix $C$ (obtained from the fitting matrix $B$), POPSTAR and dc2 were used to solve three $p$-median problems \eqref{opt_p_median}, one for each suite. The three suites reported in Tables~\ref{tab_suite_free}-\ref{tab_suite_lock2} are the best found by POPSTAR and dc2 for each $p$-median problem. Given the cost matrix, POPSTAR was able to find each suite in under 7 minutes, while dc2 was able to find each suite in under 3 minutes. POPSTAR was used with the default parameters \texttt{-graspit 32 -elite 10} for the first two suites and with the non-default parameters \texttt{-graspit 64 -elite 20} for the third suite. dc2 was used with the non-default parameters \texttt{-t12 -L -M -B0 -R50 rank:5 \_rank:5} for all three suites. POPSTAR was used with non-default parameters for the third suite, because POPSTAR used with the default parameters found a solution suboptimal (total inner volumed shipped is 27,377,482) to the one found by dc2. Lagrangian relaxation \cite{daskin2013networkChap6,daskin2015p,kwon2019juliaChap10}, as discussed in ``Solving the $p$-Median Problem" earlier in this section, determined the percentage relative optimality gaps of the first, second, and third suites to be 1.287\%, 1.101\%, and 1.087\%, respectively. The percentage relative optimality gap for a suite is defined as $\frac{\mathrm{UB}-\mathrm{LB}}{\mathrm{UB}}\cdot 100\%$, where $\mathrm{UB}$ is the total inner volume shipped by that suite and $\mathrm{LB}$ is the greatest lower bound of the minimum cost for that suite's $p$-median problem found by Lagrangian relaxation.

In Tables~\ref{tab_suite_free}-\ref{tab_suite_lock2}, the column labeled ``\% Liquid Void Volume Shipped" is the total empty volume shipped by a box divided by the total inner volume shipped by that box, and then scaled by 100\%. The empty volume for a shipment shipped in a box is measured as the box inner volume minus the shipment's liquid volume (sum of all the shipment's carton outer volumes). The total empty volume shipped by a box is the sum of the empty volumes over all the shipments shipped in that box. The total inner volume shipped by a box is the box's inner volume multiplied by the number of shipments shipped in that box. In the captions, the percentage liquid void volume shipped by a suite is the total empty volume shipped by the suite divided by the total inner volume shipped by the suite, scaled by 100\%.

Algorithm~\eqref{alg_bsr} was implemented in Julia v0.6.4 \cite{bezanson2017julia,Julia_online} using JuMP v0.18.5 \cite{DunningHuchetteLubin2017} to formulate the fitting MILP \eqref{eq_orient}-\eqref{eq_lbb}. CPLEX v12.10.0 was used to solve the fitting MILPs, while POPSTAR and dc2 were used to solve the $p$-median problems.   The software ran under the operating system Ubuntu 18.04.4 LTS (Bionic Beaver) on an Intel Core i7-3930K CPU @ 3.20 GHz with 6 physical cores (12 logical cores with hyper-threading) and 32GB RAM.

\begin{table}[h!]
	\center
		\begin{tabular}{ |c||c|c|c|c|c| } 
			\hline
			\# & ID & Inner Dimensions & Inner Volume & \shortstack{\% of Packable \\ Shipments Shipped} & \shortstack{\% Liquid Void \\ Volume Shipped} \\
			\hline \hline
			1 & 536 & $\left(11,7,4\right)$ & 308 & 35.30\% & 73.33\% \\ \hline 
			2 & 1426 & $\left(13,10,6\right)$ & 780 & 19.49\% & 58.03\% \\ \hline 
			3 & 2124 & $\left(19,14,5\right)$ & 1330 & 11.65\% & 57.54\% \\ \hline
			4 & 2705 & $\left(16,12,10\right)$ & 1920 & 9.54\% & 43.88\% \\ \hline
			5 & 3502 & $\left(32,19,5\right)$ & 3040 & 5.53\% & 59.76\% \\ \hline
			6 & 3688 & $\left(20,14,12\right)$ & 3360 & 6.51\% & 40.88\% \\ \hline
			7 & 4274 & $\left(25,17,11\right)$ & 4675 & 4.82\% & 40.90\% \\ \hline
			8 & 4854 & $\left(31,18,12\right)$ & 6696 & 3.68\% & 43.05\% \\ \hline
			9 & 5099 & $\left(26,20,16\right)$ & 8320 & 2.14\% & 40.37\% \\ \hline
			10 & 5284 & $\left(40,20,16\right)$ & 12800 & 1.35\% & 50.15\% \\ \hline
	\end{tabular}
	\caption{Best 10 box suite minimizing total inner volume shipped, found by POPSTAR and dc2. The boxes are listed by increasing inner volume. The total inner volume shipped by this suite is 26,917,098, which is within 1.287\% of the minimum possible total inner volume shipped. The percentage liquid void (box inner - liquid) volume shipped by this suite is 48.89\%. POPSTAR found this solution in 235.7[s] using the default parameters \texttt{-graspit 32 -elite 10}. dc2 found this solution in 168.2[s] using the non-default parameters \texttt{-t12 -L -M -B0 -R50 rank:5 \_rank:5}.}
	\label{tab_suite_free}
\end{table}

\begin{table}[h!]
	\center
	\begin{tabular}{ |c||c|c|c|c|c| } 
		\hline
		\# & ID & Inner Dimensions & Inner Volume & \shortstack{\% of Packable \\ Shipments Shipped} & \shortstack{\% Liquid Void \\ Volume Shipped} \\
		\hline \hline
		1 & 255 & $\left(10,6,3\right)$ & 180 & 25.51\% & 70.25\% \\ \hline 
		2 & \textbf{958} & $\mathbf{\left(12,7,6\right)}$ & 504 & 16.13\% & 59.83\% \\ \hline 
		3 & 1544 & $\left(18,12,4\right)$ & 864 & 15.82\% & 60.78\% \\ \hline
		4 & 2254 & $\left(15,12,8\right)$ & 1440 & 11.90\% & 48.33\% \\ \hline
		5 & 3132 & $\left(19,13,10\right)$ & 2470 & 9.92\% & 47.27\% \\ \hline
		6 & 3447 & $\left(31,19,5\right)$ & 2945 & 4.66\% & 62.28\% \\ \hline
		7 & 4021 & $\left(23,16,11\right)$ & 4048 & 6.72\% & 41.95\% \\ \hline
		8 & 4770 & $\left(25,18,14\right)$ & 6300 & 5.23\% & 43.62\% \\ \hline
		9 & 5082 & $\left(34,20,12\right)$ & 8160 & 2.52\% & 49.33\% \\ \hline
		10 & 5284 & $\left(40,20,16\right)$ & 12800 & 1.59\% & 47.66\% \\ \hline
	\end{tabular}
	\caption{Best 10 box suite locking the boldfaced box with ID 958 and minimizing total inner volume shipped, found by POPSTAR and dc2. The boxes are listed by increasing inner volume. The total inner volume shipped by this suite is 27,224,072, which is within 1.101\% of the minimum possible total inner volume shipped. The percentage liquid void (box inner - liquid) volume shipped by this suite is 49.47\%. POPSTAR found this solution in 205.9[s] using the default parameters \texttt{-graspit 32 -elite 10}. dc2 found this solution in 157.3[s] using the non-default parameters \texttt{-t12 -L -M -B0 -R50 rank:5 \_rank:5}.}
	\label{tab_suite_lock1}
\end{table}

\begin{table}[h!]
	\center
	\begin{tabular}{ |c||c|c|c|c|c| } 
		\hline
		\# & ID & Inner Dimensions & Inner Volume & \shortstack{\% of Packable \\ Shipments Shipped} & \shortstack{\% Liquid Void \\ Volume Shipped} \\
		\hline \hline
		1 & 509 & $\left(11,9,3\right)$ & 297 & 33.88\% & 73.66\% \\ \hline 
		2 & \textbf{958} & $\mathbf{\left(12,7,6\right)}$ & 504 & 10.38\% & 53.36\% \\ \hline 
		3 & \textbf{1918} & $\mathbf{\left(16,12,6\right)}$ & 1152 & 19.89\% & 58.67\% \\ \hline
		4 & 2807 & $\left(17,12,10\right)$ & 2040 & 10.79\% & 47.00\% \\ \hline
		5 & 3105 & $\left(32,19,4\right)$ & 2432 & 4.83\% & 65.61\%  \\ \hline
		6 & 3651 & $\left(20,15,11\right)$ & 3300 & 7.23\% & 42.81\% \\ \hline
		7 & 4213 & $\left(25,18,10\right)$ & 4500 & 4.67\% & 43.96\% \\ \hline
		8 & 4774 & $\left(31,17,12\right)$ & 6324 & 4.39\% & 43.81\% \\ \hline
		9 & 5099 & $\left(26,20,16\right)$ & 8320 & 2.51\% & 41.58\% \\ \hline
		10 & 5284 & $\left(40,20,16\right)$ & 12800 & 1.44\% & 51.42\% \\ \hline
	\end{tabular}
	\caption{Best 10 box suite locking the two boldfaced boxes with IDs 958 and 1918 and minimizing total inner volume shipped, found by POPSTAR and dc2. The boxes are listed by increasing inner volume. The total inner volume shipped by this suite is 27,370,900, which is within 1.087\% of the minimum possible total inner volume shipped. The percentage liquid void (box inner - liquid) volume shipped by this suite is 49.74\%. POPSTAR found this solution in 402.8[s] using the non-default parameters \texttt{-graspit 64 -elite 20}. dc2 found this solution in 142.0[s] using the non-default parameters \texttt{-t12 -L -M -B0 -R50 rank:5 \_rank:5}.}
	\label{tab_suite_lock2}
\end{table}

\section{Fitting MILP} \label{sec_fitting_MILP}
\paragraph{Introduction} 
Can $n \in \mathbb{N}$ 3D rectangular cartons, with positive outer lengths, widths, and heights $\left\{ \left(p_i,q_i,r_i\right) \right\}_{i=1}^n$, where $\left(p_i,q_i,r_i\right) \in \mathbb{R}_{> 0}^3$ for $i \in \llbracket n \rrbracket$, fit in a box with positive inner length, width, and height $\left(x,y,z\right) \in \mathbb{R}_{> 0}^3$, permitting orthogonal rotations of each carton? When packed into the box, it is assumed that a carton's edges must be parallel to those of the box, so that there are 6 possible orthogonal rotations of each carton. A right-handed orthogonal 3D Cartesian coordinate system is introduced in the box's frame, as depicted in Figure~\ref{fig_orientations}. The $X-$axis is parallel to the box length $x$, the $Y-$axis is parallel to the box width $y$, and the $Z-$axis is parallel to the box height $z$. The axes intersect orthogonally at $\left(0,0,0\right)$, which coincides with the left-back-bottom (lbb) corner of the box. Left to right is along the $X-$axis with 0 on the left and $x$ on the right. Back to front is along the $Y-$axis with 0 at the back and $y$ at the front. Bottom to top is along the $Z-$axis with 0 at the bottom and $z$ at the top. Whether the $n$ cartons fit into the box can be determined by checking the feasibility (or satisfiability) of a set of linear inequality constraints depending on a set of continuous and binary variables. These constraints and variables form a feasibility MILP called the fitting MILP. Since any feasibility MILP is NP-complete \cite{johnson1979computers}, the fitting MILP is NP-complete. The next few paragraphs present the constraints and variables that comprise the fitting MILP. 

\paragraph{Orientation Constraints} 
For each carton $i \in \llbracket n \rrbracket$, there are 9 orientation binary variables that indicate in which of the 6 possible ways each carton is oriented in the box. $l_{Xi},l_{Yi},l_{Zi} \in \left\{0,1\right\}$ indicate whether the length dimension $p_i$ of carton $i$ is parallel to the $X-$, $Y-$, or $Z-$axis, $w_{Xi},w_{Yi},w_{Zi} \in \left\{0,1\right\}$ indicate whether the width dimension $q_i$ of carton $i$ is parallel to the $X-$, $Y-$, or $Z-$axis, and $h_{Xi},h_{Yi},h_{Zi} \in \left\{0,1\right\}$ indicate whether the height dimension $r_i$ of carton $i$ is parallel to the $X-$, $Y-$, or $Z-$axis. The reader is referred to Figure~\ref{fig_orientations} for a 3D illustration of the meaning of the orientation binary variables.  Since each carton dimension is parallel to one and only one axis, there are six linear equality constraints on the orientation binary variables. The 6 constraints are linearly dependent with rank 5, so that only 5 are needed.
\begin{equation} \label{eq_orient}
\begin{split}
l_{Xi}+l_{Yi}+l_{Zi}&=1 \quad \forall i \in \llbracket n \rrbracket, \\
w_{Xi}+w_{Yi}+w_{Zi}&=1  \quad \forall i \in \llbracket n \rrbracket, \\
h_{Xi}+h_{Yi}+h_{Zi}&=1  \quad \forall i \in \llbracket n \rrbracket, \\
l_{Xi}+w_{Xi}+h_{Xi}&=1  \quad \forall i \in \llbracket n \rrbracket, \\
l_{Yi}+w_{Yi}+h_{Yi}&=1  \quad \forall i \in \llbracket n \rrbracket, \\
l_{Zi}+w_{Zi}+h_{Zi}&=1  \quad \forall i \in \llbracket n \rrbracket.
\end{split}
\end{equation}

\begin{figure}[h]
	\centering
	\includegraphics[scale=.5]{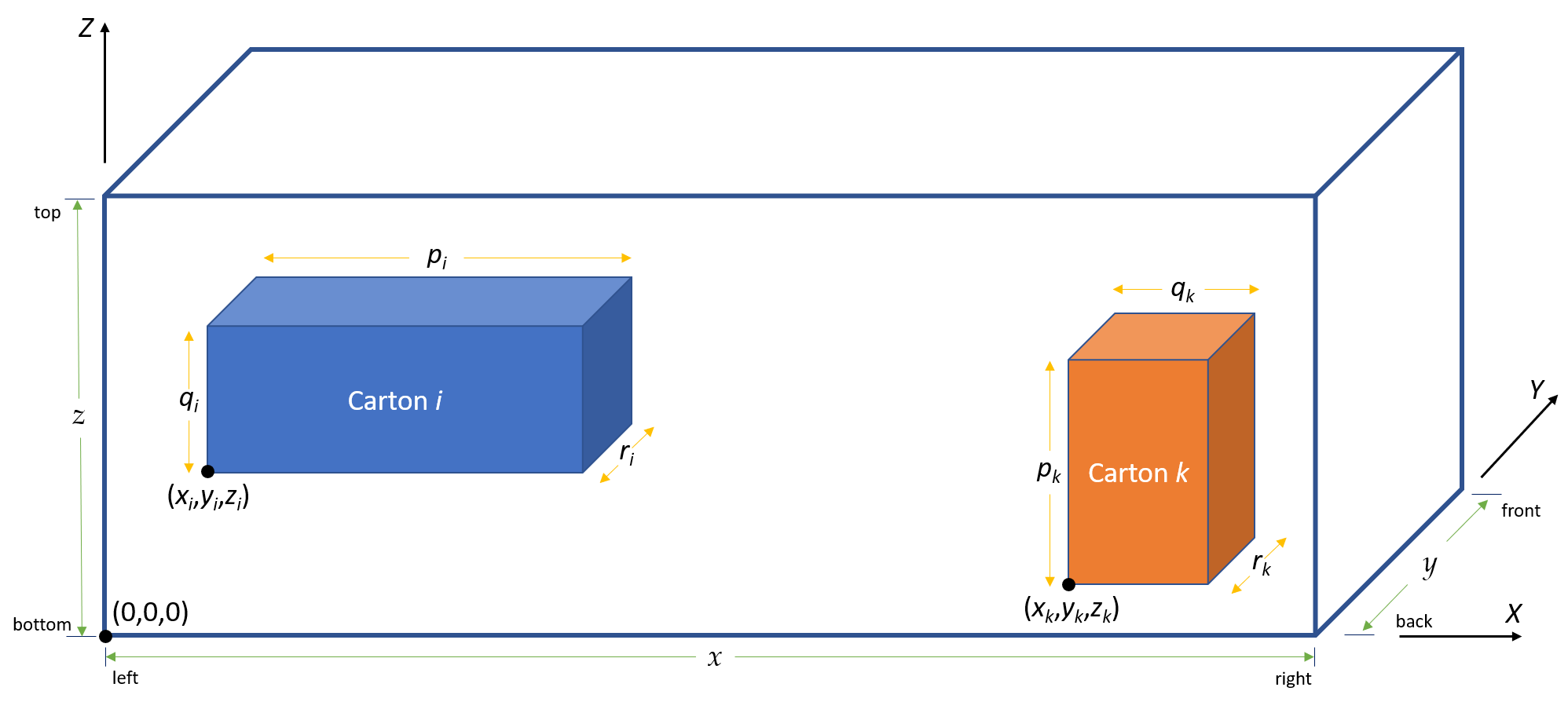}
	\caption{Orientation binary variables for cartons $i$ and $k$: $l_{Xi}=w_{Zi}=h_{Yi}=l_{Zk}=w_{Xk}=h_{Yk}=1$.}
	\label{fig_orientations}
\end{figure}

\paragraph{Containment Constraints} 
For each carton $i \in \llbracket n \rrbracket$, $\left(x_i,y_i,z_i \right) \in \mathbb{R}_{\ge 0}^3$ denotes the nonnegative coordinates of the lbb corner of carton $i$. Each carton $i \in \llbracket n \rrbracket$ must be contained in the box, which requires the following 6 linear inequality constraints.
\begin{equation} \label{eq_contain}
\begin{split}
x_ i &\ge 0  \quad \forall i \in \llbracket n \rrbracket, \\
y_i  &\ge 0 \quad \forall i \in \llbracket n \rrbracket, \\
z_i  &\ge 0 \quad \forall i \in \llbracket n \rrbracket, \\
x_i+p_i l_{Xi}+q_i w_{Xi}+r_i h_{Xi}&\le x  \quad \forall i \in \llbracket n \rrbracket, \\
y_i+p_i l_{Yi}+q_i w_{Yi}+r_i h_{Yi}&\le y  \quad \forall i \in \llbracket n \rrbracket, \\
z_i+p_i l_{Zi}+q_i w_{Zi}+r_i h_{Zi}&\le z  \quad \forall i \in \llbracket n \rrbracket.
\end{split}
\end{equation}

\paragraph{Nonoverlapping Constraints} 
Every distinct pair of cartons $i,k \in \llbracket n \rrbracket$, with $i<k$, cannot overlap. To enforce these constraints, 6 nonoverlapping binary variables indicate the relative position of pairs of cartons. $a_{ik}=1$ implies that carton $i$ is left of carton $k$, $b_{ik}=1$ implies that carton $i$ is right of carton $k$, $c_{ik}=1$ implies that carton $i$ is behind carton $k$, $d_{ik}=1$ implies that carton $i$ is in front of carton $k$, $e_{ik}=1$ implies that carton $i$ is below carton $k$, and $f_{ik}=1$ implies that carton $i$ is on top of carton $k$. The reader is referred to Figure~\ref{fig_nonoverlapping} for a 3D illustration of the meaning of the nonoverlapping binary variables. In order for cartons $i$ and $k$ to be nonoverlapping, at least one of $a_{ik}$, $b_{ik}$, $c_{ik}$, $d_{ik}$, $e_{ik}$, and $f_{ik}$ must be 1. The following 7 linear inequality constraints enforce nonoverlapping of every distinct pair of cartons $i$ and $k$.
\begin{equation} \label{eq_nonover}
\begin{split}
x_i+p_i l_{Xi}+q_i w_{Xi}+r_i h_{Xi}&\le x_k + \left(1-a_{ik} \right)x \quad \forall i,k \in \llbracket n \rrbracket, \, i<k, \\
x_k+p_k l_{Xk}+q_k w_{Xk}+r_k h_{Xk}&\le x_i + \left(1-b_{ik} \right)x \quad \forall i,k \in \llbracket n \rrbracket, \, i<k, \\
y_i+p_i l_{Yi}+q_i w_{Yi}+r_i h_{Yi}&\le y_k + \left(1-c_{ik} \right)y \quad \forall i,k \in \llbracket n \rrbracket, \, i<k, \\
y_k+p_k l_{Yk}+q_k w_{Yk}+r_k h_{Yk}&\le y_i + \left(1-d_{ik} \right)y \quad \forall i,k \in \llbracket n \rrbracket, \, i<k, \\
z_i+p_i l_{Zi}+q_i w_{Zi}+r_i h_{Zi}&\le z_k +   \left(1-e_{ik} \right) z \quad \forall i,k \in \llbracket n \rrbracket, \, i<k, \\
z_k+p_k l_{Zk}+q_k w_{Zk}+r_k h_{Zk}&\le z_i +   \left(1-f_{ik} \right) z \quad \forall i,k \in \llbracket n \rrbracket, \, i<k, \\
a_{ik}+b_{ik}+c_{ik}+d_{ik}+e_{ik}+f_{ik} &\ge 1 \quad \forall i,k \in \llbracket n \rrbracket, \, i<k.
\end{split}
\end{equation}

\begin{figure}[h]
	\centering
	\includegraphics[scale=.5]{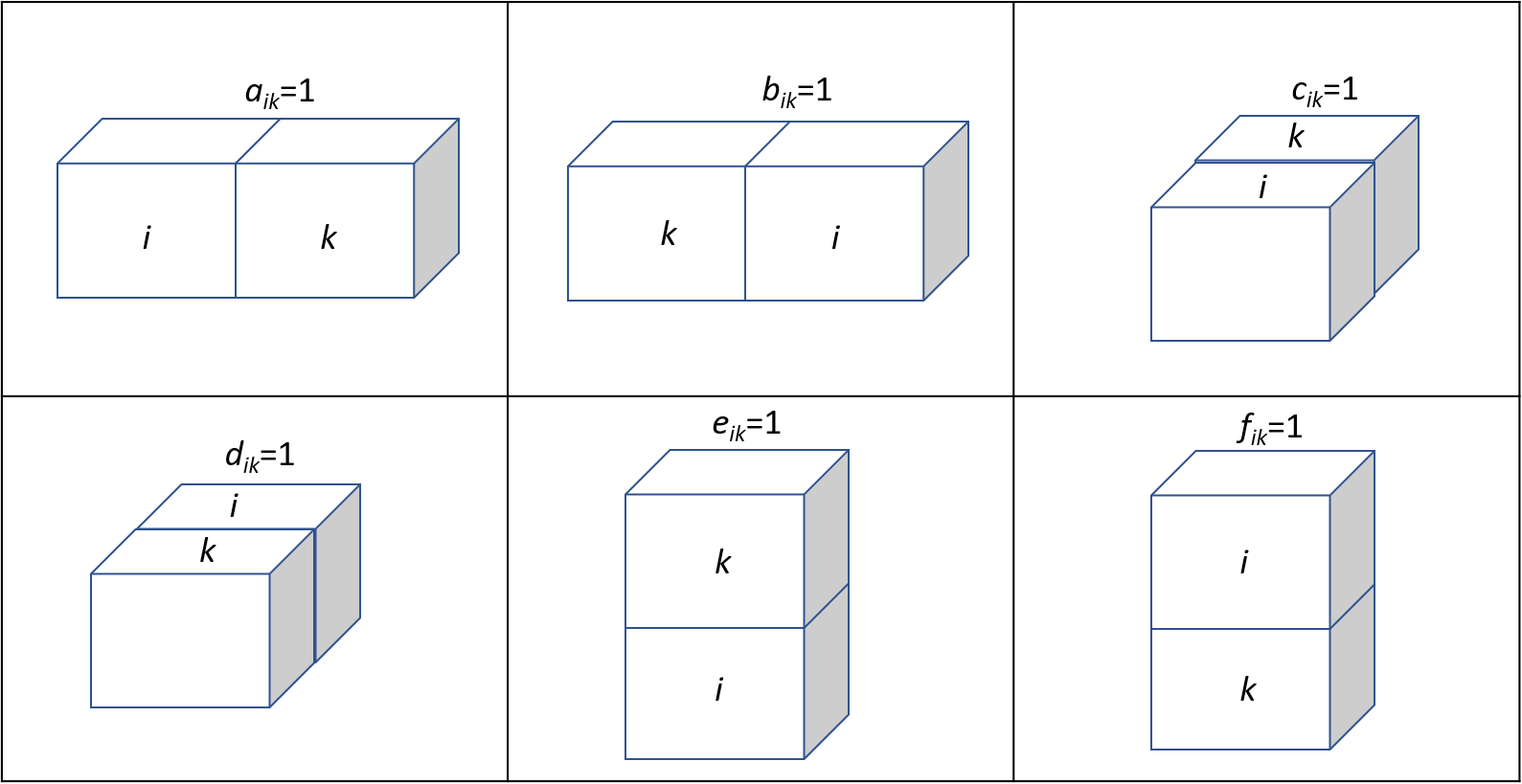}
	\caption{Illustration of the nonoverlapping binary variables.}
	\label{fig_nonoverlapping}
\end{figure}

\paragraph{Symmetry-Breaking Constraints: Identical Cartons}
Suppose that some of the $n$ cartons are identical, in the sense that they share the same sorted dimensions, and that each subset of identical cartons has been assigned consecutive indices in $\llbracket n \rrbracket$. Let $V^* \subset  \llbracket n \rrbracket$ denote the subset of carton indices in the union of all subsets of identical cartons (i.e., cartons with identical sorted dimensions). Let $V \subset V^* $ denote the subset of carton indices such that $m \in V$ if and only if $ \textproc{sort}\left(p_m,q_m,r_m\right) =  \textproc{sort}\left(p_{m+1},q_{m+1},r_{m+1}\right)$.  For $m \in V$, the $X-$coordinates (or alternatively the $Y-$ or $Z-$coordinates) of the lbb corners of identical cartons can be arranged in nondecreasing order.
\begin{equation} \label{eq_ident}
x_{m}\le x_{m+1} \quad \forall m \in V.
\end{equation}

\paragraph{Symmetry-Breaking Constraints: Carton LBB Corner in First Orthant}
Let $\beta \in \llbracket n \rrbracket$ be the index of a particular carton. For example, $\beta$ might be the index of the smallest volume carton. If $\beta \in V^*$, $\beta$ should be the smallest index of the subset of identical cartons in which $\beta$ lies, in order to be compatible with \eqref{eq_ident}.  Any feasible packing of the $n$ cartons into the box can be rearranged, through a finite sequence of reflections across the box's 3 inner half-planes, to realize a feasible packing such that the lbb corner of the carton with index $\beta$  is located in the box's first orthant $\left\{ \left(u,v,w\right) \in \mathbb{R}_{\ge 0}^3 \colon 0 \le u \le \frac{x}{2}, 0 \le v \le \frac{y}{2}, 0 \le w \le \frac{z}{2} \right\}$. 
\begin{equation} \label{eq_lbb}
\begin{split}
x_{\beta}&\le \frac{x}{2}, \\
y_{\beta}&\le \frac{y}{2}, \\
z_{\beta}&\le \frac{z}{2}.
\end{split}
\end{equation}

\paragraph{Comments on the Constraints}
The orientation \eqref{eq_orient}, containment \eqref{eq_contain}, and nonoverlapping \eqref{eq_nonover} constraints  are given in \cite{chen1995analytical,tsai2006global,tsai2015global}. References \cite{tsai2006global,tsai2015global,lin2017superior,hu2017solving,huang2018global,truong2019product} provide alternative formulations of these constraints, however the author found that a MILP solver is able to determine feasibility faster using the constraints \eqref{eq_orient}, \eqref{eq_contain}, and \eqref{eq_nonover} compared to the other constraint formulations. The symmetry-breaking constraints \eqref{eq_ident} and \eqref{eq_lbb} are new and have not appeared in the literature before and should benefit formulations of related packing problems such as the knapsack container loading problem (KCLP) \cite{pisinger2002heuristics,moura2005grasp,parreno2008maximal,parreno2010neighborhood,fanslau2010tree,gonccalves2012parallel,ramos2016physical,ramos2016container,brinker2016optimization}, the three-dimensional bin packing problem (3D-BPP) \cite{martello2000three}, and the three-dimensional open-dimension rectangular packing problem (3D-ODRPP) \cite{huang2018global}. The symmetry-breaking constraints tend to help the MILP solver determine feasibility faster by reducing the number of feasible solutions and thereby reducing the size of the search tree. Altogether, the constraints \eqref{eq_orient}-\eqref{eq_lbb} comprise the fitting MILP. The fitting MILP consists of $3n$ nonnegative continuous variables, $3n(n-1)+9n$ binary variables, $5n$ linear equality constraints, and $\frac{7}{2}n(n-1)+6n+\left| V \right|+3$ linear inequality constraints. 

\paragraph{Special Packing Constraints}
Some cartons must be packed in special ways, in which case the special packing constraints must be enforced without conflicting with the symmetry-breaking constraints \eqref{eq_ident}-\eqref{eq_lbb}. For example, some cartons cannot be stacked on top of other cartons, so that the $Z-$coordinates of their lbb corners must equal 0, in which case the constraint $z_i=0$ must be added to the fitting MILP for each such carton $i$ and the third constraint $z_{\beta} \le \frac{z}{2}$ in \eqref{eq_lbb} must be removed since a feasible packing cannot be reflected across the vertical half-plane. In ``Modifications to Handle Height-Oriented and Bottom-Resting Cartons" in Section~\ref{sec_algorithm}, such cartons were called bottom-resting (BR).

As another example, some cartons must be packed vertically, so that their height dimensions must be parallel to the box's $Z-$axis, in which case the constraint $h_{Zi}=1$ must be added to the fitting MILP for each such carton $i$. In ``Modifications to Handle Height-Oriented and Bottom-Resting Cartons" in Section~\ref{sec_algorithm}, such cartons were called height-oriented (HO). If there is at least one carton in the shipment that must be height-oriented, then the definition of identical cartons given earlier  must be revised in order to construct the subset $V$ for the symmetry-breaking constraints \eqref{eq_ident}. A pair of cartons $i,j \in \llbracket n \rrbracket$, with $i \ne j$, is identical if and only if either of the following conditions is satisfied:
\begin{enumerate}[(i)]
\item they are both not height-oriented and $ \textproc{sort}\left(p_i,q_i,r_i\right) =  \textproc{sort}\left(p_{j},q_{j},r_{j}\right)$
\item they are both height-oriented, $ \textproc{sort}\left(p_i,q_i\right) =  \textproc{sort}\left(p_{j},q_{j}\right)$, and $r_i = r_j$.
\end{enumerate}
With this new definition of identical cartons, it is still assumed that each subset of identical cartons has been assigned consecutive indices in $\llbracket n \rrbracket$ and $V^* \subset  \llbracket n \rrbracket$ denotes the subset of carton indices in the union of all subsets of identical cartons. In addition, $V \subset V^*$ denotes the subset of carton indices such that $m \in V$ if and only if cartons $m$ and $m+1$ are identical in the new sense.

As a third example, stability may be required for cartons which do not rest on the box's bottom, which requires additional constraints \cite{truong2019product} and elimination of the third constraint $z_{\beta} \le \frac{z}{2}$ in \eqref{eq_lbb} since a feasible stable packing cannot necessarily be reflected across the vertical half-plane to generate another feasible stable packing (reflecting a stable packing across the vertical half-plane may result in an unstable packing).

\paragraph{Solving the Fitting MILP}
A third-party MILP solver must be used to solve the fitting MILP. There are many MILP solvers available, but only Gurobi \cite{Gurobi_online}, CPLEX \cite{CPLEX_online}, and CBC \cite{CBC_online} are mentioned here. Gurobi and CPLEX are regarded as the best available MILP solvers, though they are commercial and require an expensive license for non-academic use. As of this writing, there is a free edition of CPLEX that solves MILPs having less than 1000 variables and 1000 constraints, which means that it can be used to solve the fitting MILP for shipments with less than 16 cartons (though the author's tests showed that it worked for shipments with less than 18 cartons). CBC is regarded as the best available free MILP solver, though its performance is quite inferior to any of the commercial MILP solvers.

\paragraph{Benchmarking I} 15,000 synthetic shipments, consisting of 1 to 8 cartons (i.e., no foldable items), were fit into 5,284 candidate boxes using CPLEX v12.10.0 and Gurobi v9.0.0 with and without the symmetry-breaking constraints \eqref{eq_ident}-\eqref{eq_lbb}. Of the 15,000 synthetic shipments, 20.88\% (3,132 shipments) consist of 4 or more cartons, for which the fitting MILP was solved by one of the MILP solvers since brute force fitting algorithms were used for shipments consisting of 1, 2, or 3 cartons. 173,066 fitting MILPs had to be solved for each of the eight MILP solver and constraint combinations. MILPs taking more than 5 seconds (s) to solve were terminated early. Table~\ref{tab_benchmark1} shows the run times (measured in seconds) and the number of MILPs terminated early due to the 5[s] time limit (TL) for each of the eight combinations. Table~\ref{tab_benchmark1} also shows the speedup and the percentage change in the number of MILPs terminated early due to the 5[s] time limit realized by using CPLEX v12.10.0 instead of Gurobi v9.0.0 and by using the symmetry-breaking constraints \eqref{eq_ident}-\eqref{eq_lbb}. Table~\ref{tab_benchmark1} shows that CPLEX v12.10.0 is faster than Gurobi v9.0.0, but does not necessarily terminate fewer MILPs early due to the 5[s] time limit. Table~\ref{tab_benchmark1} also shows that the symmetry-breaking constraints \eqref{eq_ident}-\eqref{eq_lbb} yield faster run times. The symmetry-breaking constraints \eqref{eq_ident}-\eqref{eq_lbb} terminate fewer MILPs early due to the 5[s] time limit for CPLEX but not for Gurobi. The symmetry-breaking constraint \eqref{eq_lbb} alone does not improve run times much; however, \eqref{eq_lbb} in concert with \eqref{eq_ident} gives a significant improvement over \eqref{eq_ident} alone. The benchmarking results in Table~\ref{tab_benchmark1}  were obtained using the same set of shipments, set of candidate boxes, software (Julia v0.6.4 and JuMP v0.18.5), operating system (Ubuntu 18.04.4 LTS (Bionic Beaver)), and hardware (an Intel Core i7-3930K CPU @ 3.20 GHz with 6 physical and 12 logical cores and 32GB RAM) described earlier in ``Numerical Experiments" in Section~\ref{sec_algorithm}. As described there, the set of shipments and the set of candidate boxes are publicly available on Mendeley Data \cite{https://doi.org/10.17632/f2bnnnm5zc.3}.

\begin{table}[h!]
	\resizebox{\columnwidth}{!}{
		\begin{tabular}{ |c||c|c||c| } 
			\hline
			& CPLEX v12.10.0 & Gurobi v9.0.0 & \shortstack{CPLEX Speedup /\\ \%$\Delta$ \# MILP TL} \\
			\hline \hline
			\eqref{eq_orient}-\eqref{eq_nonover} & 7490.2[s] / 129 & 8463.8[s] / 42 & 1.13 / 207.1\% \\ \hline 
			\eqref{eq_orient}-\eqref{eq_nonover} \& \eqref{eq_ident} & 5806.3[s] / 105 & 7964.0[s] / 129 & 1.37 / -18.6\% \\ \hline
			\eqref{eq_orient}-\eqref{eq_nonover} \& \eqref{eq_lbb} & 7187.2[s] / 116 & 8289.8[s] / 50 & 1.15 / 132.0\% \\ \hline
			\eqref{eq_orient}-\eqref{eq_nonover} \& \eqref{eq_ident}-\eqref{eq_lbb} & 5047.4[s] / 55 & 6944.0[s] / 76 & 1.38 / -27.6\% \\ \hline \hline
			\eqref{eq_ident} Speedup / \%$\Delta$ \# MILP TL & 1.29 / -18.6\% & 1.06 / 207.1\% & \\ \hline
			\eqref{eq_lbb} Speedup / \%$\Delta$ \# MILP TL & 1.04 / -10.1\% & 1.02 / 19.0\% & \\ \hline
			\eqref{eq_ident}-\eqref{eq_lbb} Speedup / \%$\Delta$ \# MILP TL & 1.48 / -57.4\% & 1.22 / 81.0\% & \\ \hline
	\end{tabular} }
	\caption{Comparison of CPLEX v12.10.0 and Gurobi v9.0.0 with and without the symmetry-breaking constraints \eqref{eq_ident}-\eqref{eq_lbb}. Both run times and number of MILPs terminated early due to the 5[s] time limit are compared. CPLEX v12.10.0 is faster than Gurobi v9.0.0, but does not necessarily terminate fewer MILPs early due to the 5[s] time limit. The symmetry-breaking constraints \eqref{eq_ident}-\eqref{eq_lbb} yield faster run times. The symmetry-breaking constraints \eqref{eq_ident}-\eqref{eq_lbb} terminate fewer MILPs early due to the 5[s] time limit for CPLEX but not for Gurobi.}
	\label{tab_benchmark1}
\end{table}

\paragraph{Benchmarking II} 20,425 historical customer shipments, consisting of less than 13 cartons, were fit into 4,922 candidate boxes using CPLEX v12.10.0 and Gurobi v9.0.0 with and without the symmetry-breaking constraints \eqref{eq_ident}-\eqref{eq_lbb}. Of the 20,425 historical customer shipments, 26\% (5,330 shipments) consist of 4 or more cartons, for which the fitting MILP was solved by one of the MILP solvers since brute force fitting algorithms were used for shipments consisting of 0, 1, 2, or 3 cartons. Slightly under 676,000 fitting MILPs had to be solved for each of the eight MILP solver and constraint combinations. MILPs taking more than 5 seconds (s) to solve were terminated early. Table~\ref{tab_benchmark2} shows the run times (measured in seconds) and the number of MILPs terminated early due to the 5[s] time limit (TL) for each of the eight combinations. Table~\ref{tab_benchmark2} also shows the speedup and the percentage reduction in the number of MILPs terminated early due to the 5[s] time limit afforded by using CPLEX v12.10.0 instead of Gurobi v9.0.0 and by using the symmetry-breaking constraints \eqref{eq_ident}-\eqref{eq_lbb}. Table~\ref{tab_benchmark2} shows that CPLEX v12.10.0 is faster than Gurobi v9.0.0, resulting in fewer MILPs terminated early due to the 5[s] time limit. Table~\ref{tab_benchmark2} also shows that the symmetry-breaking constraints \eqref{eq_ident}-\eqref{eq_lbb} yield faster run times, resulting in fewer MILPs terminated early due to the 5[s] time limit. The symmetry-breaking constraint \eqref{eq_lbb} alone does not improve run times much; however, \eqref{eq_lbb} in concert with \eqref{eq_ident} gives a bit of improvement over \eqref{eq_ident} alone. The benchmarking results in Table~\ref{tab_benchmark2}  were obtained using the same software (Julia v0.6.4 and JuMP v0.18.5), operating system (Ubuntu 18.04.4 LTS (Bionic Beaver)), and hardware (an Intel Core i7-3930K CPU @ 3.20 GHz with 6 physical and 12 logical cores and 32GB RAM) described earlier in ``Numerical Experiments" in Section~\ref{sec_algorithm}. Unfortunately, the set of historical customer shipments and the set of candidate boxes used to generate Table~\ref{tab_benchmark2} are confidential to Target Corporation and cannot be released publicly.

\begin{table}[h!]
	\resizebox{\columnwidth}{!}{
		\begin{tabular}{ |c||c|c||c| } 
			\hline
			& CPLEX v12.10.0 & Gurobi v9.0.0 & \shortstack{CPLEX Speedup /\\ \%$\Delta$ \# MILP TL} \\
			\hline \hline
			\eqref{eq_orient}-\eqref{eq_nonover} & 259252.0[s] / 31110 & 334261.9[s] / 40837 & 1.29 / -23.8\% \\ \hline 
			\eqref{eq_orient}-\eqref{eq_nonover} \& \eqref{eq_ident} & 231698.4[s] / 27606 & 283056.1[s] / 32600 & 1.22 / -15.3\% \\ \hline
			\eqref{eq_orient}-\eqref{eq_nonover} \& \eqref{eq_lbb} & 256932.6[s] / 30658 & 336187.4[s] / 40615 & 1.31 / -24.5\% \\ \hline
			\eqref{eq_orient}-\eqref{eq_nonover} \& \eqref{eq_ident}-\eqref{eq_lbb} & 220688.7[s] / 25832 & 272121.7[s] / 30902 & 1.23 / -16.4\% \\ \hline \hline
			\eqref{eq_ident} Speedup / \%$\Delta$ \# MILP TL & 1.12 / -11.3\% & 1.18 / -20.2\% & \\ \hline
			\eqref{eq_lbb} Speedup / \%$\Delta$ \# MILP TL & 1.01 / -1.5\% & 0.99 / -0.5\% & \\ \hline
			\eqref{eq_ident}-\eqref{eq_lbb} Speedup / \%$\Delta$ \# MILP TL & 1.17 / -17.0\% & 1.23 / -24.3\% & \\ \hline
		\end{tabular} }
		\caption{Comparison of CPLEX v12.10.0 and Gurobi v9.0.0 with and without the symmetry-breaking constraints \eqref{eq_ident}-\eqref{eq_lbb}. Both run times and number of MILPs terminated early due to the 5[s] time limit are compared. CPLEX v12.10.0 is faster than Gurobi v9.0.0, resulting in fewer MILPs terminated early due to the 5[s] time limit. The symmetry-breaking constraints \eqref{eq_ident}-\eqref{eq_lbb} yield faster run times, resulting in fewer MILPs terminated early due to the 5[s] time limit.}
		\label{tab_benchmark2}
\end{table}

\paragraph{An Additional Application} Aside from being used for box suite recommendation, an online retailer can also use the fitting MILP to recommend a box from its box suite for shipping a customer's order. Given a suite of boxes and a customer's order, the fitting MILP can be used to determine into which boxes the order fits, from which the box having minimum total shipping (shipping plus material) cost can be selected.

\section{Summary \& Future Work}
This paper offers an algorithm that recommends an optimal suite of shipping boxes subject to being able to 1) lock specific boxes in the suite and 2) pack certain items that must be height-oriented and/or bottom-resting. The algorithm assumes that shipped items are either foldable (in which case they are modeled to be liquid) or rigid (in which case they are modeled as 3D rectangular cartons). If not height-oriented, the algorithm assumes that a 3D rectangular carton must be oriented in 1 of 6 possible ways when packed into a shipping box, so that its edges are parallel to those of the box. The fitting problem is formulated and solved using MILP. New symmetry-breaking constraints are introduced that lower the computation time required to solve each fitting MILP. By solving the fitting MILPs, the algorithm determines the costs of shipping a set of shipments into a set of candidate boxes. Then, given the cost matrix, the algorithm solves a $p$-median problem  to select the minimum cost subset of $p$ boxes.

An avenue for further investigation is to model additional physically accurate constraints, such as cargo stability, in the fitting problem \cite{bortfeldt2013constraints,paquay2016mixed,junqueira2012three,junqueira2012mip,junqueira2012optimization,junqueira2013optimization,junqueira2017solving,truong2019product}. In future work, instead of using the computationally expensive MILP, it may be possible to solve the fitting problem more rapidly by using metaheuristics that solve the KCLP \cite{pisinger2002heuristics,moura2005grasp,parreno2008maximal,parreno2010neighborhood,fanslau2010tree,gonccalves2012parallel,ramos2016physical,ramos2016container,brinker2016optimization} or by using supervised machine learning (e.g., by training a neural network on a set of shipments and a fine set of candidate boxes, using the results of the fitting MILP or KCLP metaheuristics as truth). Another approach to reduce computation times is to use stratified random sampling \cite{stroehmer2012repac,levy2013sampling} to obtain a smaller sample of historical shipments, which would reduce the number of fitting problems and the size of the $p$-median problem that must be solved.

Finally, since it is logistically challenging and expensive for an online retailer to change its box suite, the sensitivity of the optimal box suite to changes in the price of corrugate, carrier rate tables, and online customer demand should be investigated and understood. 

\section*{Acknowledgements}
This research was funded by Target Corporation and by the Institute for Mathematics and its Applications at the University of Minnesota, Twin Cities. The author thanks his Target colleagues Kaveh Khodjasteh and Neil Witte for their leadership and organization, Chinmay Jethwa and Sunita Venkatachalam for providing historical shipment data, Brian Ager for his leadership and organization and for providing copious data and input parameters, and Jake Streich for his expertise in packaging engineering. Brian Ager suggested the idea of optimizing the box suite based on a statistically significant, randomly sampled subset of the previous year's customer shipments. 
Francisco Casas B. provided advice on using, improved, and corrected errors in his software dc2.

\printbibliography

\end{document}